\documentclass[12pt]{amsart}

\usepackage{amsfonts}
\usepackage{amsrefs}
\usepackage{amssymb}
\usepackage{amsmath}
\usepackage{amsthm}
\usepackage{verbatim}
\usepackage{esint} 
\usepackage{booktabs}
\usepackage{color}

\newcommand{\Z}{{\mathbb Z}}
\newcommand{\R}{{\mathbb R}}

\renewcommand{\div}{\mathrm{div}}

\def\XXint#1#2#3{{\setbox0=\hbox{$#1{#2#3}{\int}$ }
\vcenter{\hbox{$#2#3$ }}\kern-.6\wd0}}
\def\lbb{\mbox{$\,$---}}
\def\Yint#1{\mathchoice
{\XXint\displaystyle\textstyle{#1}}%
{\XXint\textstyle\scriptstyle{#1}}%
{\XXint\scriptstyle\scriptscriptstyle{#1}}%
{\XXint\scriptscriptstyle\scriptscriptstyle{#1}}%
\!\iint}
\def\fiint{\Yint\lbb}

\normalbaselines
\DeclareMathOperator *{\diam}{diam}
\DeclareMathOperator *{\dist}{dist}

\newtheorem{theorem}{Theorem}[section]
\newtheorem{lemma}[theorem]{Lemma}
\newtheorem{corollary}[theorem]{Corollary}
\newtheorem{proposition}[theorem]{Proposition}
\newtheorem{definition}{Definition}[section]

\theoremstyle{definition}

\begin{document}

\title[Elliptic operators satisfying Carleson condition]{Boundary value problems for elliptic operators satisfying Carleson condition}

\author{Martin Dindo\v{s}}
\address{School of Mathematics, \\
The University of Edinburgh and Maxwell Institute of Mathematical Sciences, Edinburgh UK}
\email{M.Dindos@ed.ac.uk}

\author{Jill Pipher}
\address{Dept. of Mathematics, \\ 
 Brown University, US}
\email{jill\_pipher@brown.edu}

\maketitle

{\centering\footnotesize This paper is dedicated to Carlos Kenig on the occasion of his 70th birthday.\par}

\begin{abstract} In this paper we present in concise form recent results, with illustrative proofs, on solvability of the $L^p$ Dirichlet, Regularity and Neumann problems for scalar elliptic equations on Lipschitz domains with coefficients satisfying a variety of Carleson conditions. More precisely, with $L=\mbox{div}(A\nabla)$, we assume the matrix $A$ is elliptic and satisfies a natural Carleson condition either in the form that ($|\nabla A(X)|\lesssim \mbox{dist}(X,\partial\Omega)^{-1}$ and $|\nabla A|(X)^2\mbox{dist}(X,\partial\Omega)\,dX$) or 
$\mbox{dist}(X,\partial\Omega)^{-1}\left(\mbox{osc}_{B(X,\delta(X)/2)}A\right)^2\,dX$ is a Carleson measure.

We present two types of results, the first is the so-called \lq\lq small Carleson" case where, for a given $1<p<\infty$, we prove solvability of the three considered boundary value problems under assumption the Carleson norm of the coefficients and the Lipschitz constant of the considered domain is sufficiently small. The second type of results (\lq\lq large Carleson") relaxes the constraints to any Lipschitz domain and to the assumption that the Carleson norm of the coefficients is merely bounded. In 
this case we have $L^p$ solvability for a range of $p$'s in a subinterval of $(1,\infty)$.

At the end of the paper we give a brief overview of recent results on domains beyond Lipschitz such as uniform domains or chord-arc domains.
\end{abstract}

\section{Introduction}

In this survey paper we provide an overview of solvability of various boundary value problems for real elliptic partial differential equations, focusing
on those with coefficients satisfying a natural Carleson condition described below. 
We further focus on the developments of the elliptic theory in the setting of Lipschitz domains, which is the context of
most of the authors' own contributions in this area, for several reasons. In the first place, the Lipschitz domain theory inspired the questions about this class of operators and their geometry is naturally connected to the Carleson condition on the coefficients. Second, while there have been 
striking developments of the elliptic theory on domains satisfying much weaker geometric conditions (chord arc, uniform), the Lipschitz domain
setting presents enough challenges to showcase many of the new ideas required to investigate and solve elliptic boundary value problems. 
Finally, there are two types of Carleson measure conditions on the coefficients - one defined for gradients, and another defined in terms of the 
oscillation of the coefficients. While the former condition has now been successfully treated, in many cases, on rougher domains than Lipschitz, the latter condition does not
generalize so readily to such domains.
That said,
it is emerging that the gradient Carleson measure condition is natural for the elliptic theory in domains satisfying these weaker 
geometric conditions and we briefly describe some of these extraordinary advances in the last section of this paper.
Our primary objective here is to illuminate the main ideas necessary to solve these particular
 Dirichlet, Regularity and Neumann problems in Lipschitz domains in an accessible manner, in its simplest yet illustrative instantiation, and
 in a single manuscript.

\subsection{Real valued elliptic PDEs}

Let $n\ge 2$ and $A(X)=(a_{ij}(X))$ be an $n\times n$ real matrix
with bounded coefficients defined for $X\in\Omega$, where
$\Omega\subset {\mathbb R}^n$ is an open,
connected set. (We will specify further assumptions on $\Omega$ a
bit later).

We are going to assume that $A$ is {\it elliptic} which  means that $A$ is {\it uniformly positive definite}.
That is for some constant $\lambda>0$ we have that
\begin{equation}
\lambda|\xi|^2\le \sum_{i,j=1}^n a_{ij}(X) \xi_i
\xi_j,\qquad\text{ for all $X\in \Omega$ and $\xi \in {\mathbb
R}^n$.}\label{ellipticity}
\end{equation}
The constant $\lambda$ is called {\it the ellipticity constant} of
and operator $L$ defined as follows:
\begin{equation}
Lu=\mbox{div}(A\nabla u)=\sum_{i,j=1}^n
\partial_i(a_{ij}\partial_j u). \label{Lop}
\end{equation}

We also denote by $\Lambda$ the $L^\infty$ norm of the matrix $A$. 
The most classical example of such operator is the flat Laplacian
on ${\mathbb R}^n$, in this case $A=I$ for all $X$.\vglue2mm

We shall make no assumption on whether matrix $A$ is symmetric or not, in the whole paper we allow
matrix $A$ to be {\bf non-symmetric}.\vglue2mm

We also note that the concept of ellipticity can be also defined for elliptic PDE with complex coefficients as well as elliptic systems. However, we shall not explore these directions further here, instead an interested reader can look at papers \cite{DPcompl}, \cite{DPreg} and \cite{DPex} for scalar complex coefficients elliptic PDEs or \cite{DHM}, \cite{Dsystems} and \cite{DLP} for elliptic systems where the same Carleson condition as in this manuscript is considered.

\subsection{Domains and Boundary Value Problems}

The domains in which one might solve these boundary value problems will 
require some constraints on the boundary.
We shall consider
Dirichlet problems with data in $L^p(\partial \Omega)$, and in the classical Sobolev space
$H^{1,p}(\partial \Omega)$, as well as the Neumann problem with $L^p$ data
$1<p<\infty$. In order to define spaces $L^p(\partial \Omega)$ and
$H^{1,p}(\partial \Omega)$ the domains must be of locally
finite perimeter; that is, the $n-1$ Hausdorff measure of $B(0,R)\cap
\partial\Omega$ is finite for $R<\infty$. Moreover, the space
$H^{1,p}(\partial \Omega)$ and the Neumann problem  require a well
defined outer normal at almost every boundary point
$\partial\Omega$. Hence, certain natural geometric assumptions have to be made about the set $\partial\Omega$.
\vglue1mm

For simplicity we present the results on Lipschitz domains, which are locally graphs of Lipschitz functions. As we note later in the paper, there have
been advances in the past several years in solving boundary value problems for more general classes of domains, also requiring more
general notions of Sobolev spaces on the boundary. With our focus on the conditions defining the coefficients of the operator, as opposed to the most
general geometric conditions possible on the domain, we can present many of the ideas and methods that illustrate the novelties required to 
solve these problems.
\vglue1mm

The plan for the rest of the paper is as follows. To start, the rest of the introduction is devoted to some key definitions and statements of the Dirichlet, Regularity, and Neumann problems with $L^p$ data.

In the second section, we state the main results and motivate the Carleson condition on the coefficients of the operators we are considering.
In section three, we give some background results and introduce elliptic measure. Section four presents results on the Dirichlet problem, and how they are 
connected to comparability of nontangential maximal function and square function estimates. The Regularity problem is discussed in section five and the Neumann problems in 
section six. Section seven is a brief overview of results on domains satisfying weaker geometric conditions, including the most recent state-of-the-art results.
Some of the delicate issues arising in bounding nontangential maximal functions by square functions are relegated to the appendix (section eight).

\begin{definition}
$\Z \subset \R^n$ is an $\ell$-cylinder of diameter $d$ if there
exists an orthogonal coordinate system $(x,t)$  with $x\in\mathbb R^{n-1}$ and $t\in\mathbb R$ such that
\[
\Z = \{ (x,t)\; : \; |x|\leq d, \; -2\ell d \leq t \leq 2\ell d \}
\]
and for $s>0$,
\[
s\Z:=\{(x,t)\;:\; |x|<sd, -2\ell d \leq t \leq 2\ell d \}.
\]
\end{definition}

\begin{definition}\label{DefLipDomain}
$\Omega\subset \R^n$ is a Lipschitz domain with Lipschitz
`character' $(\ell,N,C_0)$ if there exists a positive scale $r_0$ and
at most $N$ $\ell$-cylinders $\{{\Z}_j\}_{j=1}^N$ of diameter $d$, with
$\frac{r_0}{C_0}\leq d \leq C_0 r_0$ such that\vglue2mm

\noindent (i) $8{\Z}_j \cap {\partial\Omega}$ is the graph of a Lipschitz
function $\phi_j$, $\|\nabla\phi_j \|_\infty \leq \ell \, ;
\phi_j(0)=0$,\vglue2mm

\noindent (ii) $\displaystyle {\partial\Omega}=\bigcup_j ({\Z}_j \cap {\partial\Omega}
)$,

\noindent (iii) $\displaystyle{\Z}_j \cap \Omega \supset \left\{
(x,t)\in\Omega \; : \; |x|<d, \; \mathrm{dist}\left( (x,t),{\partial\Omega}
\right) \leq \frac{d}{2}\right\}$.

\noindent (iv) Each cylinder $\displaystyle{\Z}_j$ contains points from $\Omega^c={\mathbb R^n}\setminus\Omega$.\vglue1mm

\noindent We say that domain $\Omega$ is $C^1$ if all functions $\phi_j$ above are not only Lipschitz but also continuously differentiable.
\end{definition}

\noindent{\it Remark.} If the scale $r_0$ is finite, that is $r_0<\infty$ then the domain $\Omega$ from the definition above will be a bounded Lipschitz domain, i.e., the set $\Omega$ in $\mathbb R^n$ will be bounded.\vglue1mm

However, we shall also allow the scale $r_0$ to be infinite, in such case since $\displaystyle{\Z}=\mathbb R^n$
we are simply in the situation that in some coordinate system $\Omega$ can be written as
$$\Omega = \{(x,t): t > \phi(x)\}\quad\mbox{ where $ \phi(x):\mathbb R^n \rightarrow \mathbb R$ is a Lipschitz function.}$$

Hence $\Omega$ is an unbounded Lipschitz domain.

\begin{definition}\label{dgamma}
A cone of aperture $a > 0$ is a non-tangential approach region for $Q
\in \partial \Omega$ of the form
\[\Gamma_{a}(Q)=\{X\in \Omega: |X-Q|\le (1+a)\;\;{\rm dist}(X,\partial \Omega)\}.\]
\end{definition}

For ease of notation, and when there is no need for the specificity, we shall omit the dependence on the aperture of the cones in the definitions of the square function and nontangential maximal 
functions below.

\begin{definition}\label{d1.4}
The square function of a function $u$ defined on $\Omega$, relative to the family of cones $\{\Gamma(Q)\}_{Q \in \partial\Omega}$, is
\[S(u)(Q) =\left( \iint_{\Gamma(Q)} |\nabla u (X)|^2
\delta(X)^{2-n} dX \right)^{1/2}\] at each $Q \in \partial
\Omega$. For any $1<p<\infty$ we define the $p$-adapted square function by 
\[S_p(u)(Q) =\left( \iint_{\Gamma(Q)} |\nabla u (X)|^2|u(X)|^{p-2}
\delta(X)^{2-n} dX \right)^{1/p}\] at each $Q \in \partial
\Omega$. The
non-tangential maximal function relative to $\{\Gamma(Q)\}_{Q \in \partial\Omega}$ is
$$N(u)(Q) = \sup_{X\in \Gamma(Q)} |u(X)|$$
at each $Q \in \partial\Omega$
We also define the following variant of the non-tangential
maximal function:
\begin{equation}\label{NTMaxVar} \widetilde{N}(u)(Q) )
=\sup_{X\in\Gamma(Q)}\left(\fiint_{B_{{\delta(X)}/{2}}(X)}|u(Y)|^2\,d
Y\right)^{\frac{1}{2}}.
\end{equation}
\end{definition}

When we want to emphasize dependance of square or nontangential maximal functions on the particular cone $\Gamma_a$ we shall write $S_a(u)$, $S_{p,a}(u)$ or $N_a(u)$. Similarly, if we consider cones truncated at a certain height $h$ we shall use the notation $S^h(u)$, $S^h_a(u)$, $N^h(u)$ or $N^h_a(u)$. In general, the particular choice of the aperture $a$ does not matter, as operators with different apertures give rise to comparable $L^p$ norms.

We recall the definition of $L^p$ solvability of the Dirichlet problem. 
When an operator $L$ is uniformly elliptic, the Lax-Milgram lemma can be applied and guarantees the existence of weak solutions.
That is, 
given any $f\in \dot{B}^{2,2}_{1/2}(\partial\Omega)$, the homogenous space of traces of functions in $\dot{W}^{1,2}(\Omega$), there exists a unique $u\in \dot{W}^{1,2}(\Omega$) such that $Lu=0$ in $\Omega$ and ${\rm Tr}\,u=f$ on $\partial\Omega$. These \lq\lq  energy solutions" are used to define the solvability of the $L^p$ Dirichlet, Regularity and Neumann problems.\vglue1mm

We are now ready to formulate the three main boundary value
problems we would like to consider.

\begin{definition} Let $1<p\le\infty$.
The Dirichlet problem with data in $L^p(\partial \Omega, d\sigma)$
is solvable (abbreviated $(D)_{p}$) if for every $f\in \dot{B}^{2,2}_{1/2}(\partial\Omega) \cap  L^p(\partial\Omega)$ the weak solution $u$ to the problem $Lu=0$ with
continuous boundary data $f$ satisfies the estimate
\begin{equation}
\label{DPe}
\|N(u)\|_{L^p(\partial \Omega, d\sigma)} \lesssim \|f\|_{L^p(\partial \Omega, d\sigma)}.
\end{equation}
The implied constant depends only the operator $L$, $p$, and the
Lipschitz norm of $\varphi$. 
\end{definition}

As we assume we are on a Lipschitz domain, for almost every $Q\in\partial \Omega$ there is a well defined notion of $n-1$-dimensional hyperplane tangential to the surface $\partial\Omega$ at $Q$. We define $\nabla_T f$ for a boundary function $f:\partial\Omega\to \mathbb R$ to be a vector consisting of directional derivatives of $f$ w.r.t. directions in this tangential hyperplane at a given boundary point.

\begin{definition}\label{DefRpcondition}
Let $1<p<\infty$. The regularity problem with boundary data in
$H^{1,p}(\partial\Omega)$ is solvable (abbreviated $(R)_{p}$), if
for every $f\in \dot{B}^{2,2}_{1/2}(\partial\Omega)$ with 
$\nabla_T f \in L^{p}(\partial\Omega),$
the weak solution $u$ to the problem
\begin{align*}
\begin{cases}
Lu &=0 \quad\text{ in } \Omega\\
u|_{\partial B} &= f \quad\text{ on } \partial \Omega
\end{cases}
\end{align*}
satisfies
\begin{align}
\nonumber \quad\|\widetilde{N}(\nabla u)\|_{L^p(\partial
\Omega)}\lesssim \|\nabla_T f\|_{L^{p}(\partial\Omega)}.
\end{align}
The implied constant depends only the operator $L$, $p$,
and the Lipschitz norm of $\varphi$. 
\end{definition}

\begin{definition}\label{NPDefNpcondition}
Let $1<p<\infty$. The Neumann problem with boundary data in
$L^p(\partial\Omega)$ is solvable (abbreviated $(N)_{p}$), if for
every $f\in L^p(\partial\Omega)\cap \dot{B}^{2,2}_{-1/2}(\partial\Omega)$ with the property that
$\int_{{\partial\Omega}} fd\sigma=0$, the weak solution $u$ to the problem
\begin{align*}
\begin{cases}
Lu &=0 \quad\text{ in } \Omega\\
A\nabla u\cdot \nu &= f \quad\text{ on }
\partial \Omega
\end{cases}
\end{align*}
satisfies
\begin{align}
\nonumber \quad\|\widetilde{N}(\nabla u)\|_{L^p(\partial
\Omega)}\lesssim \|f\|_{L^{p}(\partial\Omega)}.
\end{align}
Again, the implied constant depends only the operator $L$, $p$,
and the Lipschitz norm of $\varphi$.  Here $\nu$ is the outer
normal to the boundary ${\partial\Omega}$. The sense in which $A\nabla u\cdot \nu = f$ on $\partial \Omega$ 
is that 
$$\iint_{\Omega} A\nabla u. \nabla \eta \,\,dX = \int_{\partial \Omega} f \eta \,d\sigma,$$
for all $\eta \in C_0^{\infty}(\mathbb R^n).$
\end{definition}

\noindent{\it Remark.} In the three definitions above we always ask for the corresponding non-tangential estimate (for $\|{N}( u)\|_{L^p(\partial
\Omega)}$ in the case of Dirichlet problem and $\|\widetilde{N}(\nabla u)\|_{L^p(\partial
\Omega)}$ in the case of Regularity and Neumann problems) to only hold for energy solutions.  The reason why this is enough is that the space such as $L^p(\partial\Omega)\cap \dot{B}^{2,2}_{-1/2}(\partial\Omega)$ is dense in  $L^p(\partial\Omega)$ and hence the solution operator then uniquely continuously extends to the whole 
$L^p(\partial\Omega)$. Hence it is enough to verify that an estimate like \eqref{DPe} on any dense subset of $L^p$.

A further important question is, assuming solvability as above, in what sense is the boundary datum attained. An answer to this is given in the Appendix of \cite{DPcompl}. 
 
For any $f\in L^p(\partial \Omega)$ the corresponding solution $u$ constructed by the continuous extension of the operator originally defined on a dense subset of $L^p$ attains the datum $f$ as its boundary values in the following sense.
Consider the average $\tilde u:\Omega\to \mathbb R$ defined by
$$\tilde{u}(x)=\fiint_{B_{\delta(x)/2}(x)} u(y)\,dy,\quad \forall x\in \Omega.$$
Then 
\begin{equation}\label{ntconv}
f(Q)=\lim_{x\to Q,\,x\in\Gamma(Q)}\tilde u(x),\qquad\text{for a.e. }Q\in\partial\Omega,
\end{equation}
where the a.e. convergence is taken with respect to the ${\mathcal H}^{n-1}$ Hausdorff measure on $\partial\Omega$. 
In fact, \eqref{ntconv} holds with $u(x)$ replacing $\tilde u(x)$, since solutions are H\"older continuous.
However for gradients of solutions, the  nontangential convergence holds, but only in the sense of  \eqref{ntconv}.
That is, defining
$${\tilde \nabla u}(x)=\fiint_{B_{\delta(x)/2}(x)} \nabla u(y)\,dy,\quad \forall x\in \Omega,$$
the argument of \cite{DPcompl} yields that

\begin{equation}
\nabla u(Q) =\lim_{x\to Q,\,x\in\Gamma(Q)}\tilde \nabla u(x),\qquad\text{for a.e. }Q\in\partial\Omega,
\end{equation}

It follows that all three problems have well-defined boundary values at a.e. boundary point.


\subsection{Carleson measures and oscillation}

\begin{definition} Let $\Omega$ be as above.
For $Q\in\partial\Omega$, $X\in \Omega$ and $r>0$ we write:
\begin{align*}
\Delta_r(Q) &= \partial \Omega\cap B_r(Q),\,\,\,\qquad T(\Delta_r) = \Omega\cap B_r(Q),\\
\delta(X) &=\text{\rm dist}(X,\partial \Omega).
\end{align*}
\end{definition}

\begin{definition}\label{cmeasure}
Let $T(\Delta_r)$ be the Carleson region associated to a surface
ball $\Delta_r$ in ${\partial\Omega}$, as defined above. A measure $\mu$ in $\Omega$ is
Carleson if there exists a constant $C$ such that 
\begin{equation}\label{e1.7}
\mu(T(\Delta_r))\le C \sigma (\Delta_r).
\end{equation}
The best possible $C$ is the Carleson norm and will denoted by $\|\mu\|_{Carl}$. The notation $\mu \in \mathcal C$ means that the measure  $\mu$ is Carleson. We also define a notion of vanishing Carleson measure which is a measure $\mu$ such that the best constant in \eqref{e1.7} goes to zero for balls $r\le r_0$ when we let $r_0\to 0+$.\vglue1mm 
\end{definition}

\begin{definition}\label{osc} For a function $f:\Omega\to\mathbb R$ we denote by
 $\text{osc}_B f$ for a nonempty set $B\subset\Omega$ to be
 the usual oscillation of a function
$f$ over a set $B$ which is
$$\sup_{x,y\in B}|f(x)-f(y)|.$$
\end{definition}

\

\section{Statements of main results}
\setcounter{equation}{0}

In this section we present the current state of knowledge concerning
results for the three boundary value problems on Lipschitz domains that we outlined  above. \vglue1mm

So far the only
assumptions we have made on the coefficients are that the
coefficients are {\it bounded, measurable} and satisfy the {\it
ellipticity condition}. However, examples will show that ellipticity alone is not enough to obtain solvability.

\begin{theorem} (\cite{CFK}) There exists a bounded measurable matrix $A$ on a
unit disk $D\subset{\mathbb R}^2$ satisfying the ellipticity
condition such that the Dirichlet problem $(D)_{p}$, the
Regularity problem $(R)_{p}$ and the Neumann problem $(N)_{p}$ are
not solvable for any $p\in (1,\infty)$.
\end{theorem}

The examples come from conformal considerations. For Dirichlet and Regularity problems, the counterexample
 is immediate given the existence of a solution $u$ on $D$ such
that $u\ne 0$ but $u\big|_{\partial\Omega}=0$ almost everywhere
with respect to the usual one dimensional Hausdorff measure on
$\partial D$.
Counterexamples to solvability of the Neumann problem in two dimensions follow easily from this as well via \eqref{bmatrix} which we shall discuss later.
\vglue2mm

The theorem above indicates that extra assumptions on smoothness of
coefficients will be required if we want to proceed with our program. The results stated below fall into two categories which we shall informally name \lq\lq small Carleson" and \lq\lq large Carleson".\vglue1mm

The \lq\lq small Carleson" results are results where we choose an arbitrary $p\in(1,\infty)$ and would like to know under what assumptions on Carleson norm of coefficients and the Lipschitz character of the domain we can solve the corresponding $L^p$ Dirichlet, Regularity of Neumann problem. An example is the following theorem which requires smallness of certain norms.

\begin{theorem}\label{LP} (\cites{DPP, DPR}) Let $1<p<\infty$ and let $\Omega\subset {\mathbb R}^n$ be a
Lipschitz domain with Lipschitz
`character' $(\ell,N,C_0)$ and a scale $r_0\in (0,\infty]$.

Let $Lu=\mbox{div}(A\nabla u)$ be a real-valued elliptic
differential operator defined on $\Omega$ with ellipticity
constant $\Lambda$ and coefficients which are such that
\begin{equation}\label{carlM}
d\mu(X)=\delta(X)^{-1}\left(\mbox{osc}_{B(X,\delta(X)/2)}A\right)^2\,dX
\end{equation}
is the density of a Carleson measure on all Carleson boxes of size
at most $r_0$ with norm $\|\mu(r_0)\|_{Carl}$. Then there exists
$\varepsilon=\varepsilon(\lambda,\Lambda,n,p)>0$ such that if
$\max\{\ell,\|\mu(r_0)\|_{Carl}\}<\varepsilon$ then the $(D)_p$ Dirichet, $(R)_p$
regularity problem and $(N)_p$ Neumann problems are
solvable.\vglue1mm

\noindent In particular, if the domain $\Omega$ is $C^1$ and bounded and
$A=(a_{ij})$ satisfies the vanishing Carleson condition, then
these boundary value problems are solvable for all $1<p<\infty.$
More generally, the conclusion of the theorem holds on bounded domains
whose boundary is locally given by a function $\phi$ such that
$\nabla \phi$ belongs to $L^{\infty} \cap $VMO.
\end{theorem}

Observe that the theorem above answers the solvability question for all three boundary value problems for a particular value of $p$ assuming smallness of the Carleson norm $\mu$ of coefficients of $L$ as well as that the boundary has sufficiently small Lipschitz norm (or be a $C^1/\mbox{\em VMO}$ domain). 
Examples (\cite{KKPT2}, for one) show that some assumption on the size of the Carleson measure norm is necessary if we want to solve the $L^p$ Dirichlet/Regularity/Neumann problems for a particular value of $p$.\medskip

This brings us to a second set of results which we call \lq\lq large Carleson". 
Here we relax the hypothesis on $\mu$ and $\Omega$ and only ask for $\mu$  defined as in Theorem \ref{LP} to be a Carleson measure (with potentially large norm) and similarly $\Omega$ can be an arbitrary Lipschitz domain. We then ask whether the three boundary value problems we consider are solvable for a certain range of $p\in (1,\infty)$. We start with the Dirichlet problem.

\begin{theorem}\label{Dirlarge} (\cite{KP}) Let $\Omega\subset {\mathbb R}^n$ be a
Lipschitz domain with  a scale $r_0\in (0,\infty]$, $n\ge 2$.

Let $Lu=\mbox{div}(A\nabla u)$ be a real-valued elliptic
differential operator defined on $\Omega$ with ellipticity
constant $\Lambda$ and coefficients which are such that
\begin{equation}\label{carlMM}
d\mu(X)=\delta(X)^{-1}\left(\mbox{osc}_{B(X,\delta(X)/2)}A\right)^2\,dX
\end{equation}
is the density of a Carleson measure on all Carleson boxes of size
at most $r_0$.

Then there exists $p_{dir}>1$ such that for all $p\in(p_{dir},\infty)$ the $L^p$ Dirichlet problem for the operator $L=\div(A\nabla\cdot)$ is solvable.
\end{theorem}

We then have the following result for the Regularity problem (in all dimensions) and the Neumann problem (in dimension 2):

\begin{theorem}\label{RNduality} (\cite{DHP}) Let $\Omega\subset {\mathbb R}^n$ be a
Lipschitz domain with  a scale $r_0\in (0,\infty]$, $n\ge 2$.

Let $Lu=\mbox{div}(A\nabla u)$ be a real-valued elliptic
differential operator defined on $\Omega$ with ellipticity
constant $\Lambda$ and coefficients which are such that
\begin{equation}\label{carlMMM}
d\mu(X)=\delta(X)^{-1}\left(\mbox{osc}_{B(X,\delta(X)/2)}A\right)^2\,dX
\end{equation}
is the density of a Carleson measure on all Carleson boxes of size
at most $r_0$.

Then there exists $p_{reg}>1$ such that for all $1<p<p_{reg}$ the $L^p$ Regularity problem for the operator $\ L=\div(A\nabla\cdot)$ is solvable. Furthermore $\frac1{p_{reg}}+\frac1{q_*}=1$ where $q_*>1$ is the number such that the $L^q$ Dirichlet problem for the adjoint operator $ L^*$ is solvable for all $q>q_*$. 

Additionally when $n=2$, there exists $p_{neum}>1$ such that for all $1<p<p_{neum}$ the $L^p$ Neumann problem for the operator $ L=\div(A\nabla\cdot)$ is solvable. Furthermore $\frac1{p_{reg}}+\frac1{q^*}=1$ where $q^*>1$ is the number such that the $L^q$ Dirichlet problem for the operator $ L_1=\div(A_1\nabla\cdot)$ with matrix $A_1=A/\det{A}$ is solvable for all $q>q^*$. 
\end{theorem}

\subsection{Block form operators}

In the special case $\Omega=\mathbb R^n_+$, the hypotheses in the results above can be simplified, as it is not always necessary for all coefficients to satisfy \eqref{carlMM}; for the Dirichlet problem this condition only needs to be imposed on the last row of the matrix $A$. For simplicity we do not state the most general results possible, instead we focus on the so-called block form case when the matrix $A$ is just
$$
A=\left[ \begin{array}{c|c}
   A_\parallel & 0 \\
   \midrule
   0 & 1 \\
\end{array}\right].
$$
and $A_\parallel=(a_{ij})_{1\le i,j\le n-1}$. We have this crucial result for the Dirichlet and Regularity problems:

\begin{theorem}\label{DirRegblock} Let $Lu=\mbox{div}_x(A_\parallel(\nabla_x u))+u_{tt}$ be a block form elliptic operator on $\mathbb R^n_+$ with bounded real-valued coefficients. Then the $L^p$ Dirichlet problem for the operator $L$ is solvable for all $1<p<\infty$.\medskip

\noindent If in addition the condition \eqref{carlMM} holds for coefficients of $A_\parallel$ then also the $L^p$ Regularity problem for the operator $L$ is solvable for all $1<p<\infty$.
\end{theorem}

We note that the Dirichlet part of this result is an observation of S. Mayboroda. The Regularity part can be found in \cite{DHP} and this has proven to be the key for solving the general Regularity problem with large Carleson coefficients.

\subsection{Motivation for the Carleson measure condition on the coefficients}\label{s2.1}

Consider the following simple case when the domain $\Omega$ is globally given as
$$\Omega=\{(x,t)\in {\mathbb R}^{n-1}\times{\mathbb R};\, t>\phi(x)\},$$
where $\phi$ is a Lipschitz functions $\|\phi\|_{Lip}<\infty$.

Let $L_0=\Delta$ be the usual flat Laplacian. In this case, the
solvability for $L_0$ on $\Omega$ of all three boundary value
problems $(D)_p$, $(R)_p$ and $(N)_p$ is known in an optimal range
of $p$. In particular $(D)_p$ is solvable for
$2-\varepsilon<p<\infty$, $(R)_p$ and $(N)_p$ for
$1<p<2+\varepsilon'$, here $\varepsilon, \varepsilon'$ are
determined by the Lipschitz norm $\|\phi\|_{Lip}<\infty$. (See \cite{Ken94}). \vglue2mm

Consider now a bijective bi-Lipschitz map $\Phi: {\mathbb R}^n_+
\to \Omega$ (${\mathbb R}^n_+$ being the upper half-space). Then
$$v=u\circ\Phi,\quad\text{solves the elliptic PDE}\quad L_1v=0\text{ on }{\mathbb
R}^n_+,$$ where
$$L_1=\text{div}(A\nabla.),\qquad\text{where: } A=(\det\Phi')(\Phi'^{-1})\Phi (\Phi'^{-1})^t.$$
It follows that the operator $L_1$ is bounded, elliptic and
moreover the boundary value problems $(D)_p$, $(R)_p$ and $(N)_p$
for operator $L_1$ on the upper half-space ${\mathbb R}^n_+$ are
solvable in the same range of $p$'s for which the corresponding
boundary value problems for $L_0$ on $\Omega$ are solvable.
\vglue2mm

There are two very natural bijective bi-Lipschitz map $\Phi$ we
could consider. The first one is the most obvious choice of the map
$\Phi$:
\begin{equation}
\Phi: {\mathbb R}^n_+ \to \Omega;\qquad (x,t)\mapsto
(x,t+\phi(x)).\label{map1}
\end{equation}
We now ask the following question: If $L_0=\Delta$, {\it what can we say about the regularity of the coefficients
of the operator} $L_1$?\vglue2mm

The answer is that we cannot say much beyond that the coefficients of $L_1$
are bounded and measurable. However, due to nature of the map
(\ref{map1}) we see that the matrix $A$ is independent of the
variable $t$, that is $A(x,t)=A(x)$.\vglue2mm

We now forget about the construction above and the fact that the
operator $L_1$ arose as a pull-back of a flat-Laplacian and ask
the following question. Assume that
\begin{equation}
\label{hyp1}\quad Lu=\mbox{div}(A\nabla
u),\qquad\text{for a matrix $A(x,t)=A(x)$ ($t$-independent).}
\end{equation}
If $A$ is also bounded and satisfies the ellipticity condition,
can we say something about the solvability of $(D)_p$, $(R)_p$ and
$(N)_p$ for $L$?\vglue2mm

There is a second natural choice of the map $\Phi$ we can consider
in the construction above due to Dahlberg, Keing, Ne\v{c}as, Stein
(see for example \cite{D} or \cite{N} and many others) defined as
\begin{equation}\label{map2}
\Phi(X)=(x,c_0t+(\theta_t*\phi)(x)),
\end{equation}
where $(\theta_t)_{t>0}$ is smooth compactly supported approximate
identity and $c_0$ can be chosen large enough (depending only on
$\|\nabla\phi\|_{L^\infty({\mathbb R}^{n-1})}$ so that $\Phi$ is
one to one.

Observe one new feature of this $\Phi$ as compared to (\ref{map1}).
Previously, an image of a level set $\{(x,t);\,t=const\}$ under
$\Phi$ is just a Lipschitz graph. However, for the map
(\ref{map2}) the image a such level set is a {\it smooth} function
for all $t>0$ due to the presence of a mollifier $\theta_t$. This
gives hope that {\it more smoothness} was preserved in the
pull-back procedure from $L_0$ to $L_1$. This indeed is the case
and the coefficient matrix $A$ has the property that
\begin{equation}\label{e5}
d\mu(x,t)=\sup \{ t|\nabla A(Y)|^2 \;:\;Y\in B_{t/2}((x,t)) \}dx\,dt,
\end{equation}
is a density of a Carleson measure in ${\mathbb R}^n_+$. This is how our condition \eqref{carlM} arises. 

Clearly, \eqref{e5} is not exactly \eqref{carlM} as in \eqref{e5} the coefficients are actually differentiable. But since oscillation of a function on a ball can be controlled by the supremum of the gradient multiplied by size of the ball  we see that \eqref{carlM} has the right scaling. We discuss below how we might pass from assuming 
\eqref{carlM} to instead assume \eqref{e5}  via mollification.\vglue1mm

Again, if $A$ satisfies \eqref{e5} or \eqref{carlM} and $A$ is also bounded and elliptic, we may ask whether 
we can say something about the solvability of $(D)_p$, $(R)_p$ and
$(N)_p$ for $L$.\vglue2mm

We have answered this question in Theorems \ref{LP}-\ref{RNduality}. The question that remains is the solvability of the Neumann problem in the large Carleson case in dimension larger then 2.\vglue1mm

We refer the reader to the papers \cites{HKMP1,HKMP2} where the $t$-independent case \eqref{hyp1} is considered.\vglue1mm

We now discuss the relationship of the Carleson condition to previously considered regularity assumptions on coefficients. Indeed, all three boundary value problems for elliptic operators have been considered under the assumption of various degrees of smoothness of the coefficients, starting
from $C^{\infty}$, $C^2$, $C^1$, $Lip$, $C^\alpha$, $\alpha>0$ and
finally to the Dini square condition (\cite{FJK})
$$\int_0^1 \frac{\omega^2(t)}{t}dt<\infty,$$
where $\omega$ is the modulus of continuity of the coefficients,
that is
$$|A(X)-A(Y)|\le \omega(|X-Y|),\qquad\text{for all }|X-Y|\le 1.$$
Slightly stronger than the Dini condition is the condition
$$\int_0^1 \frac{\sqrt{\omega(t)}}{t}dt<\infty,$$
which was shown in  \cite{MT} to allow the use of layer potentials to obtain solvability of these three boundary value problems. Our condition \eqref{carlM} contains all of these other conditions as subsets, in fact each of them actually implies that the Carleson measure of \eqref{carlM} is vanishing.

\section{Background results for elliptic equations}\label{S3}

Let $A$ be an $n\times n$ elliptic matrix that has bounded, measurable
coefficients. Consider the three boundary value problems for the
equation $Lu=\mbox{div}(A\nabla u)=0$ in $\Omega$. There are several known relationships connecting solvability of
these various problems for a given operator.  
(c.f. \cite{KP3}, \cite{S}, \cite{DK} et. all). Specifically, for 
$p\in(1,\infty)$, and for some $\varepsilon=\varepsilon(A,\Omega)>0:$

$$(D)_{p}\Longrightarrow (D)_q,\qquad \text{for all }q\in (p-\varepsilon,\infty).$$
$$(R)_{p}\Longrightarrow (R)_q,\qquad \text{for all }q\in (1,p+\varepsilon).$$
$$(N)_{p}+(R)_p\Longrightarrow (N)_q,\qquad \text{for all }q\in (1,p].$$
$$(R)_p\Longrightarrow (D^*)_{p'},\qquad \text{for $p'=p/(p-1)$}.$$

\medskip

Here $(D^*)$ is a Dirichlet problem for an adjoint operator
$L^*u=\mbox{div}(A^t\nabla u)$. There is also a partial converse
$$(D^*)_p+(R)_1\Longrightarrow (R)_{p'},\qquad \text{for $p'=p/(p-1)$}.$$
Here $(R)_1$ is a natural end-point Dirichlet problem with data in a
Hardy-Sobolev space and in particular $(R)_{p}\Longrightarrow (R)_1$ for any $p>1$.

\medskip

A second class of results concerns solvability issues for an
elliptic operator $L_1$ that is in some sense close to operator an
$L_0$ for which solvability is known. These results can be stated
as follows.

\begin{theorem}\label{T:perturbation} Consider operators $L_0$, ${L}_1$, with
${L}_k = \text{div}(A_k\nabla .)$ on a Lipschitz domain $\Omega$,
$\epsilon(x) = \left(a^{ij}_0(x) - a^{ij}_1(x)\right)_{i,j}$ and
$\mathbf{a}(x) = \sup_{z \in B_{\frac{\delta(x)}{2}}(x)}
\left|{\epsilon(z)}\right|$. Let
\begin{equation}
\sup_{Q \in \partial \Omega, r > 0} \frac{1}{\sigma(\Delta_r(Q))}
\iint_{T_r(Q)} \frac{\mathbf{a}^2(x)}{\delta(x)} \, dx=\epsilon_0 <
\infty,\label{E:perturbation}
\end{equation}
Finally assume that the $L^p$ Dirichlet problem $(D)_p$ is
solvable for the operator ${L}_0$.

There $M=M(p,L_0,\Omega)>0$ such that if $\epsilon_0<M$ then the
$L^p$ Dirichlet problem is solvable for the operator ${L}_1$.
\end{theorem}

This result can be found in \cite{D4}. An analogous result also
holds for the Regularity and Neumann problems (with extra
assumption in the case of the Neumann problem) by \cite{KP2}.

A second type of perturbation results  for Dirichlet and Regularity problems is as follows.

\begin{theorem}\label{T:perturbation2}(\cites{FKP,KP2}) Consider operators $L_0$, ${L}_1$ as in Theorem 
\ref{T:perturbation}, with \eqref{E:perturbation} finite (potentially large).

\noindent Then if for some $p\in(1,\infty)$ the $L^p$ Dirichlet problem $(D)_p$ is
solvable for the operator ${L}_0$, then there exists $q> 1$ such that 
the $L^q$ Dirichlet problem $(D)_q$ is
solvable for the operator ${L}_1$.

\noindent Similarly, if if for some $p\in(1,\infty)$ the $L^p$ Regularity problem $(R)_p$ is
solvable for the operator ${L}_0$, then there exists $q>1$ such that 
the $L^q$ Regulatity problem $(R)_q$ is
solvable for the operator ${L}_1$.
\end{theorem}

While the perturbation theory has been extended to more general domains for both the Dirichlet and Neumann problems (see section seven), 
it is not known whether such results hold for the Neumann problem.

\subsection{Elliptic measure, $A_\infty$ and $B_p$}

We
recall the definition of the elliptic measure. In \cite{LSW63} it
was proved that for every $g\in C(\partial\Omega)$ (or $C_0(\partial\Omega)$  if the domain is unbounded) there exists a
unique $u\in W^{1,2}_{loc}(\Omega)\cap C(\overline{\Omega})$ such
that $Lu=0$ in $\Omega$ and $u=g$ on $\partial \Omega$. Such
solution is also called a Peron's solution as it can be constructed by
a method introduced by Peron.\vglue2mm

Since we are in the case of a single equation, the maximum
principle applies. It implies that
$$\|u\|_{C(\overline{\Omega})}\leq
\|g\|_{C(\partial\Omega)}.$$ Thus for every fixed $X\in\Omega$ the
map defined by
$$C(\partial \Omega)\ni g\mapsto u(X)$$
is a bounded linear functional on $C(\partial \Omega)$. The Riesz
Representation Theorem implies the existence of a unique regular
Borel measure $\omega^X$ such that
$$u(X)= \int_{\partial \Omega} g(Q)\,d \omega^X(Q).$$
We will write $\omega$ instead of $\omega^{X}$ if we speak about a
fixed $X$. The particular choice of point $X$ does not matter,
since by the comparison principle we have
$$C^{-1}\omega^Y(E)\le \omega^X(E)\le C\omega^Y(E)$$
for a constant $C=C(X,Y)>0$ depending only on points
$X,Y\in\Omega$ but not on a set $E$. The measure $\omega$ is
called the {\it elliptic measure} of an operator $L$.\vglue2mm

We now make an explicit connection between the solvability of
Dirichlet problem $(D)_p$ and certain weight classes $B_p$ (sometimes also denoted $RH_p$).\vglue2mm

The reverse H\"older class $B_q$, $q>1$, is defined as the class
of all non-negative functions $k\in L^1_{loc}$ such that
$$\left(\fint_Q k^q\right)^{\frac{1}{q}}\leq C\fint_Q k$$
for all balls $Q$, where $\fint_Q k=\frac{1}{|Q|}\int_Q k$. Using
for example Lemma 1.4.2 in \cite{Ken94} one sees that (after writing $d\omega=k\,d\sigma)$:
$$(D)_p \Leftrightarrow \omega\in B_{p'}(d\sigma),\qquad p'=p/(p-1).$$

We shall denote by $A_\infty(d\sigma)$
$$A_\infty(d\sigma)=\bigcup_{p'>1}B_{p'}(d\sigma).$$

Observe that when $\omega\in A_\infty(d\sigma)$ if follow that there exists $p_0>1$ such that
 the $L^p$ Dirichlet problem $(D)_p$ is
solvable for the operator ${L}$ with elliptic measure $\omega$ for all $p>p_0$. Hence, it is therefore extremely important to know when the elliptic measure of a particular operator belongs to this class. 
An important breakthrough in our understanding when this happens is due to two papers \cite{DKP09} and \cite{KKPT}. We formulate the result in the theorem below.

\begin{theorem}\label{T:Ainfty} Consider an elliptic operator $L$, with
${L} = \text{div}(A\nabla .)$ on a Lipschitz domain $\Omega$ and let $\omega=\omega^X$ for some $X\in\Omega$ be the elliptic measure of this operator.

\noindent Then $\omega\in A_\infty(d\sigma)$ if and only if there exists a constant $C<\infty$ such that 
for all surface balls $\Delta\subset \partial\Omega$ and all Borel subsets $E\subset \Delta$ the solution $u$ to the equation $Lu=0$ in $\Omega$ with boundary datum $\chi_E$ satisfies the estimate
\begin{equation}\label{e1}
\sup_{B} |B|^{-1}\iint_{B\cap \Omega}|\nabla u(X)|^2\delta(X)\,dX\le C.
\end{equation}
Here, the supremum is taken over all balls $B$ in the ambient space centred at a boundary point and $\delta(X)$ denotes the distance of an interior point $X$ to the boundary $\partial\Omega$.
\end{theorem}

The condition \eqref{e1} in this particular formulation has become the primary tool used to prove the $A_\infty$ property and hence solvability of the $L^p$ Dirichlet problem for some $p>1$.
Unfortunately, no similar result is known for the Regularity and Neumann problems which creates more technical difficulties in arguing for solvability of both of these boundary value problems.

\section{Dirichlet problem}
\setcounter{equation}{0}

In this section we give some of the main ideas and calculations involved in the proofs of Theorems \ref{LP} and \ref{Dirlarge} for the Dirichlet problem. 
Assume for now that ${L} = \text{div}(A\nabla .)$ is a divergence-form elliptic operator on an unbounded Lipschitz domain 
$$\Omega=\{(x,t)\in {\mathbb R}^{n-1}\times{\mathbb R};\, t>\phi(x)\},$$
as in the section \ref{s2.1}.
The  pull-back transformation \eqref{map2} we have introduced there to motivate the Carleson condition  
allows us to consider the corresponding Dirichlet/Regularity or Neumann problems in
the domain ${\mathbb R}^n_+$. This is because, for $\Omega=\{(x,t):\, t>\phi(x)\}$, the pull-back map preserves the ellipticity condition and the Carleson condition on the coefficients (although the Carleson bound coefficients of  the new operator on ${\mathbb R}^n_+$ might increase and will depend on $\|\nabla\phi\|_{L^\infty}$ as well).

In the case $\Omega$ is a bounded Lipschitz domain we need to do a further localization argument which for the sake of brevity of our presentation we omit here, but which is fairly standard. An interested reader can see \cite{DPP} for details.

Hence from now on we assume that $\Omega={\mathbb R}^n_+$. We shall rename the variable $x_n$ of ${\mathbb R}^n_+$ as $t$, to distinguish it from the remaining directions. Hence we interchangeably use the notation $\partial_n=\partial_t$ for the derivative in this direction throughout the paper.

The next reduction comes in the form of replacing the Carleson condition \eqref{carlM} by the stronger condition: 
\begin{equation}\label{CCalt}
\delta(X)\left[\sup_{Y\in B(X,\delta(X)/2)}|\nabla A(Y)|\right]^2\mbox{ is a Carleson measure}.
\end{equation}
To see this, one consider a new matrix $\bar{A}$ obtained from $A$ via mollification
$\bar{A}(x,t)=(A*\eta_{t/2})(x,t)$ for a family of smooth mollifiers $(\eta_t)_{t>0}$ (for details see \cite{DPP} where this observation was made). The matrix valued function $\bar{A}$ is uniformly elliptic but now satisfies \eqref{CCalt} instead of the oscillation condition, \eqref{carlM}, that holds for $A$. In addition, we also have
\begin{equation}\label{CCpert}
\delta(X)^{-1}\left[\sup_{Y\in B(X,\delta(X)/2)}|A(Y)-\bar{A}(Y)|\right]^2\mbox{ is a Carleson measure.}
\end{equation}

By Theorem \ref{T:perturbation} we can then use solvability (which we establish below) for the operator with matrix $\bar{A}$ to deduce solvability of the Dirichlet problem for the operator with matrix $A$, since \eqref{CCpert} will have  small Carleson norm, provided \eqref{carlM} was small.

In case \eqref{carlM} has large Carleson norm, then Theorem \ref{T:perturbation2} applies and we get solvability 
of the $L^p$ Dirichlet problem for the operator with matrix $A$ for some large $p>1$. Thus matters can really be reduced to operators on $\mathbb R^n_+$ satisfying either small or large Carleson condition \eqref{CCalt}. \medskip

Below we follow \cite{DPP} and just consider $p=2$. 
Our goal is to prove that for a continuous data $f$ the estimate
\begin{equation}\label{e7}
\|N(u)\|_{L^2} \lesssim
\|f\|_{L^2}
\end{equation}
holds for an energy solution $u$ with datum $f$. We begin with the following lemma.

\begin{lemma}\label{lgeneral} Let ${L} = \text{div}(A\nabla .)$ be an elliptic operator on ${\mathbb R}^n_+$ such that the matrix $A$ satisfies \eqref{CCalt} and let $\|\mu\|_{Carl}$ be its Carleson norm.

\noindent Assume that $w:{\mathbb R}^n_+\to \mathbb R$ is such that for every boundary ball $\Delta\subset\partial {\mathbb R}^n_+$ we have $w\in W^{1,2}(T(\Delta)\cap {\mathbb R}^n_+)$.
Then the following statements hold:
\begin{equation} \label{e4.4}
\iint_{T(\Delta)} S(w)^2\,dX\le C(\|\mu\|_{Carl})\int_{2\Delta}N(w)^2\, dx- C\iint_{\mathbb R^n_+}\frac1{a_{nn}}(Lw)w\phi t\,dX,
\end{equation}
for some smooth cutoff function $\phi=1$ on $T(\Delta)$ and vanishing outside $T(2\Delta)$. Furthermore, assume that $w$ has sufficiently fast decay to zero as $(x,t)\to\infty$. Then 
\begin{eqnarray} \label{e4.5}
\label{ee1m} && \int_{\partial{\mathbb
R}^n_+}S(w)^2\,dx
\le C\int_{\partial{\mathbb R}^n_+}|w|^2
dX\\\nonumber&&\qquad+C\|\mu\|_{Carl}\int_{\partial{\mathbb
R}^n_+}N(w)^2 dx-C\iint_{{\mathbb R}^n_+}\frac1{a_{nn}}(Lw)wt\,dX.
\end{eqnarray}
\end{lemma}

\noindent{\it Remark.} A careful observation of the proof given below establishes that only coefficients $(a_{nj})_{j=1,2,\dots, n}$ of the matrix $A$ have to satisfy the Carleson condition for \eqref{e4.4}-\eqref{e4.5} and hence $\mu$ only needs to capture the Carleson norm of these coefficients.\medskip

In particular, if $u$ is an energy solution to $Lu=0$ in ${\mathbb R}^n_+$ it follow that

\begin{corollary}\label{l1} Under the same assumption of $L$ as in Lemma \ref{lgeneral} we have for any energy solution of $Lu=0$ with datum $u\big|_{\partial{\mathbb
R}^n_+}=f\in C_0^\infty(\mathbb R^{n-1})$:
\begin{eqnarray} \label{e4.6}
\label{ee1} && \int_{\partial{\mathbb R}^n_+}S(u)^2\,dx
\le C\int_{\partial{\mathbb R}^n_+}|f|^2
dX+C\|\mu\|_{Carl}\int_{\partial{\mathbb R}^n_+}N(u)^2 dx.
\end{eqnarray}
If in addition $f$ is bounded and supported on a ball $\Delta$ then
\begin{equation} \label{e4.7}
\frac1{|\Delta|}\iint_{T(\Delta)} |\nabla u|^2t\,dX\le C(\|\mu\|_{Carl})\|f\|_{L^\infty}.
\end{equation}
\end{corollary}

\begin{proof} Clearly \eqref{e4.6} follows from  \eqref{e4.5} as the last term vanishes due to $u$ being a solution, provided $u$ has sufficient decay at infinity.
In particular, we need that $(r_0)^{-1}\iint_{{\mathbb R}^{n-1}\times [r_0,2r_0]}|u|^{2} \, dX\to 0$ as $r_0\to
\infty$. This is indeed the case under our assumption but requires some extra approximation to be shown. We outline here the main idea. We consider $u_n$ to be the energy solutions to $Lu_n=0$ in $\Omega_n=\{(x,t):0<t<n\}$ with boundary datum $f$ at $t=0$ and vanishing at $t=n$. We then define $u_n=0$ for $t\ge n$. Clearly, each such $u_n$ will satisfy $(r_0)^{-1}\iint_{{\mathbb R}^{n-1}\times [r_0,2r_0]}|u_n|^{2} \, dX\to 0$ and we could show that \eqref{e4.6} still holds for constants that do not depend on $n$. That is 
\begin{eqnarray}
\label{ee1z} && \int_{\partial{\mathbb R}^n_+}S(u_n)^2\,dx
\le C\int_{\partial{\mathbb R}^n_+}|f|^2
dX+C\|\mu\|_{Carl}\int_{\partial{\mathbb R}^n_+}N(u_n)^2 dx.
\end{eqnarray}
Taking the limit $n\to\infty$ it can be shown that $u_n\to u$ where $u$ is the energy solution to $Lu=0$ on $\mathbb R^n_+$ with boundary datum $f$ at $t=0$.
Furthermore $u_n\to u$ locally uniformly on compact subsets of $\mathbb R^n_+$ and $\nabla u_n\to \nabla u$ in $L^2$ on such compact subsets. Hence \eqref{e4.6} holds for $u$ by taking limit in the inequality \eqref{ee1z}.

Also \eqref{e4.7} follows from \eqref{e4.4} and the maximum principle (the energy solutions do satisfy the maximum principle due to the decay to zero at infinity) as for any $x\in\mathbb R^{n-1}$ we have that $N(u)(x)\le \|u\|_{L^{\infty}(\mathbb R^n_+)}\le \|f\|_{L^\infty}$.
\end{proof}

\noindent {\it Proof of Lemma~\ref{lgeneral}.}
Here and below we use the summation convention. We introduce the following localisation. Let 
\begin{equation}
\phi(x,t)=\varphi(x)^2\psi(t)^2,\quad\mbox{where }\varphi\in C_0^\infty(\mathbb R^{n-1}),\quad \varphi(x)=\begin{cases}1,&\mbox{ for }x\in \Delta,\\0,&\mbox{ for }x\notin 2\Delta,\end{cases}
\end{equation}
and $\psi\in C^\infty(\mathbb R)$, $\psi(t)=\begin{cases}1,&\mbox{ for }t\le r_0,\\0,&\mbox{ for }t\ge 2r_0\end{cases}$. Here $r_0>0$ will be specified later.

We begin by
integrating by parts using the ellipticity condition. Assuming we choose $r_0$ so that $r_0\approx \mbox{diam} (\Delta)$ we see that
\begin{eqnarray}\nonumber
\iint_{T(\Delta)} |\nabla w|^2t\,dX&\le&
\iint_{{\mathbb R}^n_+}|\nabla w|^2\phi t\,dX\approx\iint_{{\mathbb
R}^n_+}\frac{a_{ij}}{a_{nn}}(\partial_iw)(\partial_jw)\phi t\,
dX =\\
\label{ee2} &-&\iint_{{\mathbb
R}^n_+}\frac1{a_{nn}}w\partial_i(a_{ij}\partial_j w) t\phi\, dX-\iint_{{\mathbb
R}^n_+}w(\partial_jw)a_{ij}\partial_i\left(\frac{\phi t}{a_{nn}}\right)\,
dX.
\end{eqnarray}

\noindent Notice that the first term in the second line contains 
$Lw=\partial_i(a_{ij}\partial_j w)$ and we no not deal with it anymore as it is as in \eqref{e4.4}. We work on the last term of \eqref{ee2}.
It is equal to

\begin{eqnarray}\label{ee3}
-\iint_{{\mathbb R}^n_+}w(\partial_jw)\frac{a_{nj}}{a_{nn}} \phi\,
dX&+&\iint_{{\mathbb
R}^n_+}w(\partial_jw)\frac{a_{ij}}{a^2_{nn}}(\partial_ia_{nn})\phi t\,
dX\\&-&\iint_{{\mathbb R}^n_+}w(\partial_jw)\frac{a_{nj}}{a_{nn}} (\partial_i\phi) t\,
dX.\nonumber
\end{eqnarray}
Consider now the first term of (\ref{ee3}). For $j=n$ we get that it is equal to

\begin{equation}
-\frac{1}2\iint_{{\mathbb R}^n_+}\partial_n(|w|^{2}\phi) \, dX+\frac{1}2\iint_{{\mathbb R}^n_+}|w|^{2}(\partial_n\phi) \, dX
=\frac12\int_{\partial{\mathbb R}^n_+}|w(\cdot,0)|^2\phi\,
dx+\frac{1}2\iint_{{\mathbb R}^n_+}|w|^{2}(\partial_n\phi) \, dX,\label{ee4}
\end{equation}
which corresponds to Dirichlet data at the boundary. For $j<n$ the
first term of (\ref{ee3}) is handled as follows. We introduce an
artificial $1$ into the term by placing $\partial_{n} t$ inside
the integral. After integration by parts we get

\begin{eqnarray}
&-&\frac{1}2\iint_{{\mathbb
R}^n_+}\partial_j(|w|^{2})\frac{a_{nj}}{a_{nn}}\phi(\partial_{n}t) \,
dX=\label{ee5}\frac12\iint_{{\mathbb
R}^n_+}\partial_{n}\left(\partial_j(|w|^2)\frac{a_{nj}}{a_{nn}}\phi\right)t
\,dX\\\nonumber&=& \frac12\iint_{{\mathbb
R}^n_+}\partial_j\partial_{n}(|w|^2)\frac{a_{nj}}{a_{nn}}\phi t \,
dX+\frac12\iint_{{\mathbb
R}^n_+}\partial_j(|w|^2)\partial_{n}\left(\frac{a_{nj}}{a_{nn}}\right)\phi
t \, dX\nonumber\\&+&\frac12\iint_{{\mathbb
R}^n_+}\partial_j(|w|^2)\frac{a_{nj}}{a_{nn}}(\partial_n \phi)
t \, dX.
\nonumber
\end{eqnarray}
The first term after the last equal sign can be
further integrated by parts and we obtain

\begin{eqnarray}\nonumber
&&\iint_{{\mathbb
R}^n_+}\partial_j\partial_{n}(|w|^2)\frac{a_{nj}}{a_{nn}}\phi t \,
dX=\\&&\qquad-\iint_{{\mathbb
R}^n_+}\partial_{n}(|w|^2)\partial_j\left(\frac{a_{nj}}{a_{nn}}\right)
\phi t \, dX-\iint_{{\mathbb
R}^n_+}\partial_n(|w|^2)\frac{a_{nj}}{a_{nn}}(\partial_j \phi)
t \, dX.\label{ee6}
\end{eqnarray}

Now we estimate terms that look similar. The second term of \eqref{ee3}, the second term of \eqref{ee5} after the equal sign and the first term of \eqref{ee6} after the equal sign can all be bounded by

\begin{equation}
 C \iint_{{\mathbb R}^n_+}|w||\nabla w||\nabla
A| \phi t \, dX.
 \label{ee7}
 \end{equation}

Here $\nabla A$ stands for either $\nabla a_{nj}$ or $\nabla
a_{nn}$. Notice also the the last term of (\ref{ee3}) is also of
this type. By Cauchy-Schwarz we get that the righthand side of
(\ref{ee7}) is less than

\begin{equation}
C\left(\iint_{{\mathbb R}^n_+}|w|^{2}|\nabla A|^2 \phi t \, dX\right)^{1/2}\left(\iint_{{\mathbb R}^n_+}|\nabla w|^2 \phi t \,
dX\right)^{1/2}.
 \label{ee8}
 \end{equation}
Using the Carleson condition on the coefficients, and the fact
that their Carleson norm \eqref{CCalt} is $\|\mu\|_{Carl}$  we get that
this can be further bounded by

\begin{equation}
 \label{ee9}
C\|\mu\|^{1/2}_{Carl}\left(\int_{2\Delta}N(w)^{2} dX\right)^{1/2} \left(\iint_{{\mathbb
R}^n_+}|\nabla w|^2\phi t \, dX\right)^{1/2}.
\end{equation}
Here we have used the fact that the support of $\phi$ at the boundary is inside $2\Delta$. Finally, after using the inequality between arithmetic and geometric means we can achieve that
\begin{equation}\label{e4.17} 
\mbox{\eqref{ee9}}\le \frac13 \iint_{{\mathbb
R}^n_+}|\nabla w|^2\phi t \, dX +C\|\mu\|_{Carl}\int_{2\Delta}N(w)^{2} dX.
\end{equation}
The first term is exactly one third of the second term on the first line of \eqref{ee2} and hence can be absorbed by it, while the second term is precisely what we need in \eqref{e4.4}.

In order to establish \eqref{e4.4} we still need to estimate a few more remaining terms.  The first term after the equal sign of \eqref{ee4} is pointwise bounded by the nontangential maximal function, i.e., $|w(x,0)|\le N(w)(x)$ and hence the integral bound follows.

The second term after the equal sign of \eqref{ee4} can be bounded by $C\int_{2\Delta}N(w)^{2} dX$. Indeed, $\partial_n \phi$ is supported in $2\Delta\times [r_0,2r_0]$ and is of size $(r_0)^{-1}$. For any $(x,t)$ inside the support we also have a pointwise bound $|w(x,t)|\le N(w)(x)$ from which the claim follows.

Consider now the last term of \eqref{ee3}, the last term of \eqref{ee5} and the last term of \eqref{ee6}. Given that $|\nabla \phi|\lesssim (\phi)^{1/2}(r_0)^{-1}$ and that matrix coefficients of $A$ are bounded they all can be estimated by
\begin{equation}\label{e4.18}
C \iint_{\mathbb R^{n-1}_+}|\nabla w||w|\frac{(\phi)^{1/2}t}{r_0}dX\le \frac13 \iint_{{\mathbb
R}^n_+}|\nabla w|^2\phi t \, dX+C'\iint_{2\Delta\times[0,2r_0]}|w|^2(r_0)^{-1}dX,
\end{equation}
using the AG inequality and the fact that $(t/r_0)\le 2$. We again absorb the first term after the equal sign into 
the second term on the first line of \eqref{ee2}, while the last term again has the bound by $C'\int_{2\Delta}N(w)^{2} dX$ for the reasons already explained above. This concludes the proof of \eqref{e4.4}.
\medskip

We now consider \eqref{e4.5}. Here for some terms we proceed differently, as we turn the local estimate we have so far into a global one. Consider a cover of the boundary $\mathbb R^{n-1}$ by non-overlapping rectangles $(\Delta_n)_{n\in\mathbb N}$, all of size $r_0$. Let $(\varphi_n)_{n\in\mathbb N}$ be a partition of unity subordinate to the enlarged rectangles $(2\Delta_n)_{n\in\mathbb N}$. Then for each cutoff function $\phi_n:=\varphi_n(x)\psi^2(t)$ where $\psi$ is again smooth and $\psi(t)=\begin{cases}1,&\mbox{ for }t\le r_0,\\0,&\mbox{ for }t\ge 2r_0\end{cases}$,
 we have that the calculation \eqref{ee2}-\eqref{e4.17} holds. We now sum over all $n\in\mathbb N$ to get
\begin{eqnarray}\label{e4.19}
&&\hskip6mm\iint_{{\mathbb R}^n_+}|\nabla w|^2\psi^2 t\,dX\le C\int_{\partial{\mathbb R}^n_+}|w(\cdot,0)|^2
dx+C\iint_{{\mathbb R}^n_+}|w|^{2}|\partial_n(\psi^2)| \, dX\\&&\nonumber -C\iint_{{\mathbb R}^n_+}\frac1{a_{nn}}(Lw)w\psi^2 t\,dX+C\|\mu\|_{Carl}\int_{\partial{\mathbb R}^n_+}N(w)^{2} dX+C\iint_{{\mathbb
R}^n_+}|\nabla w||w|\partial_n (\psi)^2t \, dX.
\end{eqnarray}

Here the first two terms after the equal sign come from \eqref{ee4}, the third one is the first term of the last line of \eqref{ee2}, the fourth term is due to \eqref{ee7}-\eqref{e4.17} and the last term comes from the least term of \eqref{ee5}. Terms such as the last term of \eqref{ee3} and last term of \eqref{ee6} are completely gone since $\sum_n\varphi_n=1$ implies that $\sum_n\partial_j\phi_n=0$ for all $j<n$.

The last term of \eqref{e4.19} is dealt with using Cauchy-Schwarz and the AG inequality in a spirit similar to what we did for \eqref{e4.18}. Using the fact that $|\partial_n\phi|\lesssim (r_0)^{-1}$ we finally obtain:

\begin{eqnarray}\label{e4.20}
&&\hskip6mm\iint_{{\mathbb R}^n_+}|\nabla w|^2\psi^2 t\,dX\le C\int_{\partial{\mathbb R}^n_+}|w(\cdot,0)|^2
dx+C(r_0)^{-1}\iint_{{\mathbb R}^{n-1}\times[r_0.2r_0]}|w|^{2} \, dX\\&&\qquad\qquad\qquad\qquad\qquad\nonumber -C\iint_{{\mathbb R}^n_+}\frac1{a_{nn}}(Lw)w\psi^2 t\,dX+C\|\mu\|_{Carl}\int_{\partial{\mathbb R}^n_+}N(w)^{2} dX.
\end{eqnarray}
From this \eqref{e4.5} follows by letting $r_0\to\infty$ as the assumption that $w$ has sufficient decay that infinity implies that the term $(r_0)^{-1}\iint_{{\mathbb R}^{n-1}\times[r_0.2r_0]}|w|^{2} \, dX$ converges to zero. \qed

The following lemma has been established in \cite{KP} via a stopping time argument. We shall prove a version of it in the appendix of this paper.

\begin{lemma}\label{l2} Consider any
operator $L$ of the form $Lu=\div A\nabla u$ on ${\mathbb R}^n_+$
with bounded elliptic coefficients $A$ such that (\ref{CCalt}) is
a Carleson measure. Then for any $p>0$ and any energy solution $Lu=0$
\begin{equation}
\label{equiv2} \|N(u)\|_{L^p(\partial{\mathbb R}^n_+)}\lesssim\|S_2(u)\|_{L^p(\partial{\mathbb R}^n_+)}.
\end{equation}
\end{lemma}

\noindent{\it Remark.} Again a careful study of the proof reveals that the Carleson condition is only required for the coefficients in the last row of the matrix $A$.\medskip

We are now ready to prove Theorem \ref{Dirlarge} as well as the Dirichlet part of Theorem \ref{LP} for $p=2$.\medskip

\noindent {\it Proof of Theorem~\ref{Dirlarge}.} It suffices to prove that the elliptic measure of our operator $L$ that satisfies \eqref{CCalt} on $\mathbb R^n_+$ belongs to $A_\infty(d\sigma)$. However by Theorem \ref{T:Ainfty} we only need to show \eqref{e1} for $f=\chi_E$, where $E\subset \Delta$ is a Borel set. This however follows from \eqref{e1} via the following consideration. We approximate $\chi_E$ by a sequence $(f_n)$ of $C_0^\infty(\mathbb R^{n-1})$ functions with support in a small enlargement of $\Delta$ and bounded by $1$. For each such $f_n$ we know that \eqref{e4.7} will hold for the energy solution $u_n$ with boundary datum $f_n$. 
\begin{equation} \label{e4.7w}
\frac1{|\Delta|}\iint_{T(\Delta)} |\nabla u_n|^2t\,dX\le C(\|\mu\|_{Carl})\|f_n\|_{L^\infty}=C(\|\mu\|_{Carl}).
\end{equation}
Also $u_n\to u$ locally uniformly in $W^{1,2}$ on compact subsets of $\mathbb R^n_+$. This allows to take the limit $n\to\infty$ to get
\begin{equation} \label{e4.7ww}
\frac1{|\Delta|}\iint_{T(\Delta)\cap\{t>\varepsilon\}} |\nabla u|^2t\,dX\le C(\|\mu\|_{Carl}),
\end{equation}
for any $\varepsilon>0$. Finally, limiting $\varepsilon\to 0$ we get the claim for $u$.
\qed\medskip

\noindent {\it Proof of Theorem~\ref{LP}.} We combine \eqref{equiv2} with \eqref{e4.6}. It follows that for $Lu=0$ with $u\big|_{\partial{\mathbb R}^n_+}=f\in L^2\cap C_0^\infty$
\begin{equation}
\label{eeeee} \|N(u)\|^2_{L^2}\le C_1\|S_2(u)\|^2_{L^2}\le CC_1\|f\|^2_{L^2}+CC_1\|\mu\|_{Carl}\|N(u)\|^2_{L^2}.
\end{equation}
It follows that if we take $\|\mu\|_{Carl}$ small enough so that $CC_1\|\mu\|_{Carl}<1/2$ we get that
\begin{equation}
\label{eeee} \|N(u)\|^2_{L^2}\le 2CC_1\|f\|^2_{L^2},
\end{equation}
giving us solvability of the $L^2$ Dirichlet problem for $f$ from a dense subset of $L^2$. This is however sufficient, see the remark after Definition \ref{NPDefNpcondition} on extending the solvability to the whole $L^2$.\qed\medskip

So far the calculation has been done for $p=2$. To consider different values of $p>1$ the key new idea is so-called $p$-adapted square function  defined in Definition \ref{d1.4} and originally introduced in \cite{DPP}.
It is then possible via similar integration by parts as above to establish
and analogue of Corollary \ref{l1}.

\begin{lemma}\label{l1b} Let $p\in(1,\infty)$. If $u$ is a bounded energy solution to $Lu=0$ in the domain ${\mathbb R}^n_+$
then we have for some $C=C(p,\lambda,\Lambda,n)>0$:
\begin{eqnarray}
\nonumber && \int_{\partial{\mathbb
R}^n_+}S_p(u)^p\,dx=\iint_{{\mathbb R}^n_+}|u|^{p-2}|\nabla
u|^2 t\,dX \le C\int_{\partial{\mathbb R}^n_+}|u|^p
dX+C\|\mu\|_{Carl}\int_{\partial{\mathbb R}^n_+}N(u)^p dx.
\end{eqnarray}
\end{lemma}
We leave this calculation to an interested reader, the case $p=2$ being a guide. This, together with an
analogue of Lemma \ref{l2}, yields solvability of $(D)_p$ for operators with coefficients having sufficiently small Carleson norm $\|\mu\|_{Carl}$.\medskip

Finally, we address briefly the block form case (c.f. Theorem \ref{DirRegblock}). As the Carleson norm of the last row of a block-form matrix $A$ has Carleson norm zero, it follows that for all $p\in(1,\infty)$ we have
\begin{eqnarray}
\nonumber && \int_{\partial{\mathbb
R}^n_+}S_p(u)^p\,dx\le C\int_{\partial{\mathbb R}^n_+}|u|^pdx.
\end{eqnarray}
Since also $N\approx S_p$ we get solvability of $(D)_p$ for the block form operators on $\mathbb R^n_+$ without imposing any conditions on the coefficients beyond boundedness and ellipticity.

\section{Regularity problem}
\setcounter{equation}{0}

We consider first the Regularity problem under the small Carleson condition on the coefficients. The result here follows  \cite{DR} where the two dimensional case was done, and then subsequently \cite{DPR}, for all dimensions.

Recalling from section \ref{S3} of this paper,  it suffices to prove solvability of the
Regularity problem for $p=2$. Solvability for other values of
$p$ is then a consequence of the fact that 
\cite{DK} has established:
$$(D^*)_p+(R)_1\Longrightarrow (R)_{p'},\qquad \text{for $p'=p/(p-1)$}.$$
It follows that if we establish $(R)_2$ for operators satisfying
small Carleson condition we conclude that $(R)_{p'}$ also holds,
since $(R)_2\Longrightarrow (R)_1$ and the solvability of
$(D^*)_p$ follows from the previous section.

Hence we prove $(R)_2$ solvability for operators with sufficiently small Carleson norm on domains with small Lipschitz character. Again, it suffices to consider an operator $L$ on $\mathbb R^n_+$ satisfying \eqref{CCalt} as we can perform the same reductions as for the Dirichlet problem we have discussed previously.

The first step is an analogue of
Corollary \ref{l1} but only for the square function of the tangential
gradient $\nabla_Tu$ of a solution $u$. Eventually, we want to control the full gradient
for which we shall use the equation that $u$ satisfies.

\begin{lemma}\label{ll1} If $u$ is a bounded energy solution to $Lu=0$ in the domain ${\mathbb R}^n_+$.
Then we have the following:
\begin{eqnarray}
&&\label{e8}\\\nonumber && \int_{\partial{\mathbb
R}^n_+}S(\nabla_T u)^2\,dx \le C\int_{\partial{\mathbb
R}^n_+}|\nabla_Tu|^2 dx+\|\mu\|_{Carl}\int_{\partial{\mathbb
R}^n_+}N(\nabla
u)^2 dx,
\end{eqnarray}
where $\|\mu\|_{Carl}$ is the Carleson norm \eqref{CCalt} of coefficients of the operator $L$.
\end{lemma}

\begin{proof}
We apply Lemma \ref{lgeneral} to  partial derivatives $v_k=\partial_k
u$, $k=1,2,\dots,n-1$.
The main difference from the previous case is that the last term
of (\ref{e4.5}) is no longer equal to zero. Hence we obtain the following:
\begin{eqnarray}
\label{e9ff} \sum_{k=1}^{n-1}\int_{\partial{\mathbb
R}^n_+}S(v_k)^2\,dx &\le&
C\sum_{k=1}^{n-1}\int_{\partial{\mathbb R}^n_+}|v_k|^2
dx+C\|\mu\|_{Carl}\int_{\partial{\mathbb
R}^n_+}N(\nabla u)^2 dx\\
\nonumber&-&C\sum_{k=1}^{n-1}\iint_{{\mathbb
R}^n_+}\frac1{a_{nn}}v_k(Lv_k) t\, dX.
\end{eqnarray}

Clearly, since $\partial_k(Lu)=0$ we see that $L(\partial_k
u)=Lv_k=[L,\partial_k]v_k$, where $[.,.]$ denotes the usual
commutator bracket. This yields that each $v_k$ is a solution of
the following auxiliary inhomogeneous equation:
\begin{equation}
\div(A\nabla v_k)=Lv_k=-\div((\partial_k A){\bf v})=\div
\vec{F_k},\label{eqDer}
\end{equation}
where the $i$-th component of the vector $\vec{F_k}$ is $(\vec{
F_k})^i=-(\partial_ka_{ij})\partial_ju=-(\partial_ka_{ij})v_j$.
\vglue2mm

Using this, the last term of (\ref{e9ff}) is

\begin{eqnarray}&&\label{e22}
\sum_{k<n}\iint_{{\mathbb
R}^n_+}\frac{v_k}{a_{nn}}\partial_i((\partial_ka_{ij})v_j) t\,
dX=-\sum_{k<n}\iint_{{\mathbb
R}^n_+}\partial_i\left(\frac{v_k}{a_{nn}}t
\right)(\partial_ka_{ij})v_j\, dX,\\\nonumber
&=&-\sum_{k<n}\iint_{{\mathbb
R}^n_+}\partial_i\left(\frac{1}{a_{nn}}
\right)(\partial_ka_{ij})v_jv_kt\, dX
-\sum_{k<n}\iint_{{\mathbb R}^n_+}\frac{1}{a_{nn}}
(\partial_ka_{ij})(\partial_i v_k)v_jt\, dX
\\\nonumber&&-\sum_{k<n}\iint_{{\mathbb R}^n_+}\frac{\partial_ka_{nj}}{a_{nn}}v_k v_j\, dX.
\end{eqnarray}
where we have integrated by parts. We note that the last term only
appears for $i=n$ as ($\partial_n(t)=1$). The first term
is bounded by $\|\mu\|_{Carl} \int_{{\mathbb
R}^n_+}N(\nabla u)^2\,d\sigma$. Here we are using the small Carleson
condition and the bound $v_j,v_k\le N(\nabla u)$. (We are omitting the localization details 
that  ensure that the integrals are finite in this sketch).  The second term is handled
exactly as (\ref{ee7}). Hence this term is (in absolute value)
smaller than
$$\frac12 \sum_{k<n}\iint_{{\mathbb R}^n_+}|\nabla v_k|^2 t \, dX+
C\|\mu\|_{Carl}\int_{\partial{\mathbb R}^n_+}N^2(\nabla
u)dx.$$  Thus as before the first term term can be
absorbed into the left-hand side of (\ref{e9ff}). Hence the only
term remaining is
$$-\sum_{k<n}\iint_{{\mathbb R}^n_+}\frac{\partial_ka_{nj}}{a_{nn}}v_k v_j\, dX=-\sum_{k<n}\iint_{{\mathbb
R}^n_+}\frac{\partial_ka_{nj}}{a_{nn}}v_k v_j\partial_n(t)\,
dX.$$ Here we have introduced an extra term $1=\partial_n(t)$
and now integrate by parts again. This gives

\begin{eqnarray}\nonumber&&\sum_{k<n}\iint_{{\mathbb
R}^n_+}\partial_n\left(\frac{1}{a_{nn}}\right)(\partial_ka_{nj})v_k
v_j t\, dX+ \sum_{k<n}\iint_{{\mathbb
R}^n_+}\frac{\partial_ka_{nj}}{a_{nn}}v_k (\partial_nv_j) t\,
dX+
\\\label{e24} &+& \sum_{k<n}\iint_{{\mathbb
R}^n_+}\frac{\partial_ka_{nj}}{a_{nn}}(\partial_nv_k) v_j t\,
dX+\sum_{k<n}\iint_{{\mathbb
R}^n_+}\frac{\partial_n\partial_ka_{nj}}{a_{nn}}v_k v_j t\,
dX.
\end{eqnarray}

The first three terms are of same type we have encountered above
and same bounds apply to them. Finally, in the last term we have
two derivatives on the coefficients (the term
$\partial_n\partial_ka_{nj}$) but only one of the derivatives is
in the normal direction since $k<n$. Hence we integrate by parts
one more time (moving the $\partial_k$ derivative). We get three more terms

\begin{eqnarray} &-& \sum_{k<n}\iint_{{\mathbb
R}^n_+}(\partial_na_{nj})\partial_k\left(\frac1{a_{nn}}\right)v_k
v_j t\, dX
\\\nonumber &-& \sum_{k<n}\iint_{{\mathbb
R}^n_+}\frac{\partial_na_{nj}}{a_{nn}}(\partial_kv_k) v_j t\,
dX-\sum_{k<n}\iint_{{\mathbb
R}^n_+}\frac{\partial_na_{nj}}{a_{nn}}v_k (\partial_kv_j) t\,
dX.
\end{eqnarray}
All these enjoy the same bounds as the terms we encountered above.
From this (\ref{e8}) follows. \vglue2mm
\end{proof}

Lemma \ref{ll1} deals with the square function estimates for
tangential directions. We have the following for the normal
derivative:

\begin{lemma}\label{ll2} Let $u$ be a solution to $Lu=0$, where $L$ is an elliptic differential operator with
bounded coefficients which are such that (\ref{CCalt}) is the
density of a Carleson measure. Then
\begin{eqnarray}
\label{e30} \int_{\partial{\mathbb
R}^n_+}S^2(\partial_nu)\,d\sigma&\le&K\left[\int_{\partial{\mathbb
R}^n_+}S^2(\nabla_Tu)\,d\sigma+\|\mu\|_{Carl}\int_{\partial{\mathbb
R}^n_+}N^2(\nabla u)\,d\sigma\right].
\end{eqnarray}
Here $K$ only depends on the ellipticity constant and dimension $n$.
\end{lemma}

\begin{proof}  We use the notation introduced above
where we denoted $v_n=\partial_n u$. Clearly
\begin{eqnarray}
\label{e31} &&\iint_{{\mathbb R}^n_+}|\nabla v_n(X)|^2
t\,dX=\iint_{{\mathbb R}^n_+}|\nabla_T v_n(X)|^2
t\,dX+\iint_{{\mathbb R}^n_+}|\partial_n
v_n(X)|^2 t\,dX\\
\nonumber&=&\iint_{{\mathbb R}^n_+}|\partial_n (\nabla_T u(X))|^2
t\,dX+\iint_{{\mathbb R}^n_+}|\partial_n v_n(X)|^2 t\,dX.
\end{eqnarray}
The first term is clearly controlled by the square function of
$\nabla_T u$ which has a bound by Lemma \ref{ll1}. It remains to deal with the second term. Since
$$|a_{nn}\partial_nv_n|^2=|\partial_n(a_{nn}v_n)-\partial_n(a_{nn})v_n|^2\le 2|\partial_n(a_{nn}v_n)|^2+2|\partial_n(a_{nn})v_n|^2.$$
We see that by the ellipticity assumption
\begin{eqnarray}
\label{e32} &&\iint_{{\mathbb R}^n_+}|\partial_n v_n(X)|^2
\delta(X)\,dX\approx\iint_{{\mathbb R}^n_+}(a_{nn}(X))^2|\partial_n
v_n(X)|^2 \delta(X)\,dX \\\nonumber &\le& 2\iint_{{\mathbb
R}^n_+}|\partial_n(a_{nn}v_n)|^2 t\,dX+2\iint_{{\mathbb
R}^n_+}|\partial_n(a_{nn})v_n|^2 t\,dX.
\end{eqnarray}
The second term (using the Carleson condition) is bounded by
$\|\mu\|_{Carl}\int_{{\partial\Omega}}N^2(\nabla u)\,dx$. We further
estimate the first term. Using the equation $u$ satisfies we see
that
$$\partial_n(a_{nn}v_n)=-\sum_{(i,j)\ne(n,n)}\partial_i(a_{ij}\partial_j u).$$
It follows that
\begin{eqnarray}
\label{e33} &&\iint_{{\mathbb R}^n_+}|\partial_n(a_{nn}v_n)|^2
t\,dX\le(n^2-1)\sum_{(i,j)\ne(n,n)}\iint_{{\mathbb
R}^n_+}|\partial_i(a_{ij}\partial_j u)|^2
t\,dX\\\nonumber&\le&2(n^2-1)\sum_{(i,j)\ne(n,n)}\left[\iint_{{\mathbb
R}^n_+}|\partial_i(a_{ij})|^2|\partial_j u|^2 t\,dX+\iint_{{\mathbb
R}^n_+}|a_{ij}|^2|\partial_i\partial_j u|^2t\,dX\right].
\end{eqnarray}
The first term here is of the same type as the last term of
(\ref{e32}) and has the same bound. Because $(i,j)\ne(n,n)$
$$|\partial_i\partial_j u|^2\le |\nabla(\nabla_T u)|^2,$$
hence the last term of (\ref{e33}) is also bounded by the square
function of $\nabla_T u$.
\end{proof}

If we combine the results of Lemma \ref{ll1} and $\ref{ll2}$ we
obtain the following inequality.

\begin{lemma}\label{ll3} Let $u$ be an energy solution to $Lu=\div
A\nabla u=0$, where $L$ is an elliptic differential operator with
bounded coefficients which are such that (\ref{CCalt}) is the
density of a Carleson measure. Then there exists $K>0$ depending only on the ellipticity
constant and dimension $n$ such
that
\begin{eqnarray}
\label{e34} &&\int_{\partial{\mathbb R}^n_+}S^2(\nabla
u)\,dx\le 
K\left[\int_{\partial{\mathbb
R}^n_+}|\nabla_Tu|^2 dx+\|\mu\|_{Carl}\int_{\partial{\mathbb
R}^n_+}N^2(\nabla
u) dx\right].
\end{eqnarray}
In particular, for such $u$:
\begin{eqnarray}
\label{e34b} &&\int_{\partial{\mathbb R}^n_+}S^2(\nabla
u)\,dx\lesssim
\int_{\partial{\mathbb R}^n_+}N^2(\nabla u) dx.
\end{eqnarray}
\end{lemma}

To establish $(R)_2$, we follow the same idea as presented in the previous section on solvability of $(D)_2$. Clearly we can conclude that $(R)_2$ will hold for sufficiently small $\|\mu\|_{Carl}$, provide we also have the following fact.

\begin{lemma}\label{ll3a} Let $L$ be as in Lemma \ref{ll3} and assume that the Carleson norm of $\|\mu\|_{Carl}$ is sufficiently small. Then
\begin{eqnarray}
\label{e34a} &&\int_{{\mathbb R}^n_+}N^2(\nabla u) dx\lesssim
\int_{\partial{\mathbb R}^n_+}S^2(\nabla u)\,dx.
\end{eqnarray}
\end{lemma}

We again briefly address the main idea of the proof in the appendix. For further details see \cite{DPR}. \medskip

Now we look at the large Carleson case. Consider now an elliptic operator $L$ on $\mathbb R^n_+$ with
bounded coefficients which are such that (\ref{CCalt}) is a Carleson measure (but not necessarily small). We aim to establish the Regularity part of the claim of Theorem \ref{RNduality} in this case. We present a recent argument from \cite{DHP}.\medskip

The argument consists of two parts, the first one of which is the reduction to special block from matrices, stated below.

\begin{theorem}\label{t1} Let $L=\div(A\nabla\cdot)$ be an operator in ${\mathbb R}^n_+$ where matrix $A$ is uniformly elliptic, with bounded real coefficients such that there exists a constant $C$
\begin{equation}\label{Carl2}
|\nabla A|^2t\,dt\,dx\qquad\mbox{is a Carleson measure, and} \qquad t |\nabla A| \le C.
\end{equation}
Suppose that for some $p>1$ the $L^p$ Regularity problem for the block form operator
\begin{equation}\label{e0a}
{ L}_0 u= \mbox{\rm div}_\parallel(A_\parallel \nabla_\parallel u)+u_{tt},
\end{equation}
(where $A_{\parallel}$ is the matrix $(a_{ij})_{1\le i,j\le n-1}$) is solvable in ${\mathbb R}^n_+$. \vglue3mm

Then we have the following: For any $1<q<\infty$ the $L^q$ Regularity problem for the operator $L$ is solvable in ${\mathbb R}^n_+$ if and only if the $L^{q'}$ Dirichlet problem for the adjoint operator $L^*$ is solvable in ${\mathbb R}^n_+$.
\end{theorem}

Assume for now that $L_0$ is such that for some $p>1$ the Regularity problem $(R)_p$ is solvable in ${\mathbb R}^n_+$. Since by Theorem \ref{Dirlarge} there exists $p_{dir}>1$ such that the Dirichlet problem for the operator $L^*$ is solvable for $p\in (p_{dir},\infty)$ (as the coefficients of $L^*$ also satisfy the large Carleson condition) it follows from Theorem \ref{t1} that the $L^q$ Regularity problem for $L$ is solvable in the interval $q\in (1,q_{reg})$, where $1/p_{dir}+1/q_{reg}=1$.\medskip

Thus Theorem \ref{t1}, together with the theorem below, implies the Regularity part of the claim of Theorem \ref{RNduality}.

\begin{theorem}\label{tblock}
Let ${ L}_0 u= \mbox{\rm div}_\parallel(A_\parallel \nabla_\parallel u)+u_{tt}$ be an operator in ${\mathbb R}^n_+$ where matrix $A_\parallel$ is uniformly elliptic $(n-1)\times(n-1)$ matrix, with bounded real coefficients such that \begin{equation}\label{Carl2m}
d\mu(X)=\delta(X)\left[\sup_{B(X,\delta(X)/2)}|\nabla A(X)|\right]^2\,dX\qquad\mbox{is a Carleson measure.}
\end{equation}
Then we have the following: For any $1<q<\infty$ the $L^q$ Regularity problem for the operator $ L_0$ is solvable in ${\mathbb R}^n_+$.
\end{theorem} 
\noindent Note that Theorem \ref{tblock} implies the second claim in Theorem \ref{DirRegblock}.\medskip

\noindent It remains to prove Theorems \ref{t1} and \ref{tblock}. We start with Theorem \ref{t1}.\medskip

\noindent {\it Proof of Theorem \ref{t1}.}
Throughout this proof, we make the assumption that $|\nabla A(x,t)|$ is bounded by a constant $M$ for all $(x,t)$.
 All the estimates established below will be independent of $M$. This assumption entails that boundary integrals like those in
 \eqref{e8c}, \eqref{e9}, and so on, are meaningful in a pointwise sense. The assumption can be removed by approximating by a matrix
 that satisfies condition \eqref{Carl2} by a sequence of matrices with bounded gradients - details can be found in
 section 7 of \cite{Dsystems}.

We start by summarising useful results from \cite{KP2}. Let us denote by $\tilde{N}_{1,\varepsilon}$ the $L^1$-averaged version of the non-tangential maximal function for {\it doubly truncated} cones. That is, for $\vec{u}:{\mathbb R}^n_+\to\mathbb R^m$, we set

$$\tilde{N}_{1,\varepsilon}(\vec{u})(Q)=\sup\left\{\fiint_{Z\in B(X,\delta(X)/2)}|\vec{u}|dZ:\, X\in\Gamma_\varepsilon(Q):=\Gamma(Q)\cap
\{X: \varepsilon < \delta(X)< 1/\varepsilon\}\right\}.$$

Lemma 2.8 of \cite{KP2}, stated below, provides a way to estimate the $L^q$ norm of $\tilde{N}_{1,\varepsilon}(\nabla F)(Q)$ via duality (based on tent-spaces).

\begin{lemma}\label{l1bb} There exists $\vec{\alpha}(X,Z)$ with $\vec{\alpha}(X,\cdot):B(X,\delta(X)/2)\to{\mathbb R}^n$ and \newline $\|\vec\alpha(X,\cdot)\|_{L^\infty(B(X,\delta(X)/2))}=1$, a nonnegative scalar function $\beta(X,Q)\in L^1(\Gamma_\varepsilon(Q))$ with $\int_{\Gamma_\varepsilon(Q)}\beta(X,Q)\,dX=1$ and a nonnegative $g\in L^{q'}(\partial{\mathbb R^n_+},d\sigma)$ with $\|g\|_{L^{q'}}=1$ such that
\begin{equation}
\left\|\tilde{N}_{1,\varepsilon}(\nabla F)\right\|_{L^q(\partial{\mathbb R^n_+},d\sigma)}\lesssim \iint_{{\mathbb R^n_+}}\nabla F(Z)\cdot \vec{h}(Z)\, dZ,
\label{e1a}
\end{equation}
where 
$$\vec{h}(Z)=\int_{\partial{\mathbb R^n_+}}\iint_{\Gamma(Q)}g(Q)\vec{\alpha}(X,Z)\beta(X,Q)\frac{\chi(2|X-Z|/\delta(X))}{\delta(X)^n}\,dX\,dQ,$$
and $\chi(s)=\chi_{(0,1)}(|s|)$.

Moreover, for any $G:{\mathbb R^n_+}\to\mathbb R$ with $\tilde{N}_{1}(\nabla G)\in L^q(\partial{\mathbb R^n_+},dx)$ we also have an upper bound
\begin{equation}
\iint_{{\mathbb R^n_+}}\nabla G(Z)\cdot h(Z)\, dZ\lesssim \left\|\tilde{N}_{1}(\nabla G)\right\|_{L^q(\partial{\mathbb R^n_+},dx)}.
\label{e2}
\end{equation}
The implied constants in \eqref{e1a}-\eqref{e2} do not depend on $\varepsilon$, only on the dimension $n$.
\end{lemma}

For the matrix $A = (a_{ij})$ as above, we let $v:{\mathbb R}^n_+\to\mathbb R$ be the solution of the inhomogenous Dirichlet problem for the operator ${L}^*$ (adjoint to $L$):
\begin{equation}
{ L}^*v=\div(A^*\nabla v)=\div(\vec{h})\mbox{ in }{\mathbb R}^n_+,\qquad v\Big|_{\partial{\mathbb R}^n_+}=0.\label{e3}
\end{equation}

Then Lemma 2.10 - Lemma 2.13 of \cite{KP2} gives us the following estimates for the nontangential maximal and square functions of $v$.

\begin{lemma}\label{l2aa} If the $L^{q'}$ Dirichlet problem is solvable for the operator ${L}^*$, where $q>1$, 
 then there exists $C<\infty$ depending on $n$, $q$, and $ L^*$, such that for any $\vec{h}$ as in Lemma \ref{l1bb} and $v$ defined by \eqref{e3} we have
\begin{equation}
\|N(v)\|_{L^{q'}(\partial{\mathbb R}^n_+,d\sigma)}+\|\tilde{N}(\delta\nabla v)\|_{L^{q'}(\partial{\mathbb R}^n_+,d\sigma)}+\|S(v)\|_{L^{q'}(\partial{\mathbb R}^n_+,d\sigma)}\le C.
\label{e4}
\end{equation}
\end{lemma}
\vglue5mm

 Let $u$ be the solution of the following boundary value problem
\begin{equation}
{L}u=\div(A\nabla u)=0\mbox{ in }{\mathbb R}^n_+,\qquad u\Big|_{\partial{\mathbb R}^n_+}=f,\label{e5xb}
\end{equation}
where we assume that $f\in \dot{W}^{1,q}(\partial{\mathbb R}^n_+\cap \dot{B}^{2,2}_{1/2}(\partial{\mathbb R}^n_+)$ for some $q>1$. Then clearly, $u\in \dot{W}^{1,2}({\mathbb R}^n_+)$ by Lax-Milgram. 

Fix $\varepsilon>0$. Our aim is to estimate $N_{1,\varepsilon}(\nabla u)$ in $L^q$ using Lemma \ref{l1bb}. Let $\vec{h}$ be as in Lemma \ref{l1bb} for $\nabla F=\nabla u$. Then since $\vec{h}\big|_{\partial{{\mathbb R}^n_+}}=0$ and $\vec{h}$ vanishes at $\infty$, we have by integration by parts
\begin{equation}\label{e6}
\|N_{1,\varepsilon}(\nabla u)\|_{L^q}\lesssim \iint_{{{\mathbb R}^n_+}}\nabla u\cdot\vec{h}\,dZ=-\iint_{{\mathbb R}^n_+} u\,\div\vec{h}\,dZ
\end{equation}
$$\hskip1cm=-\iint_{{{\mathbb R}^n_+}}u\,{ L}^*v\,dZ=-\iint_{{{\mathbb R}^n_+}}u\,\div(A^*\nabla v)\,dZ.$$
We now move $u$ inside the divergence operator and apply the divergence theorem to obtain:
\begin{equation}\nonumber
\mbox{RHS of \eqref{e6}}=-\iint_{{{\mathbb R}^n_+}}\div(uA^*\nabla v)\,dZ+\iint_{{\mathbb R}^n_+} A\nabla u\cdot\nabla v\,dZ=
\iint_{\partial{{\mathbb R}^n_+}}u(\cdot,0)a^*_{nj}\partial_jv\,dx,
\end{equation}
since
$$\iint_{{\mathbb R}^n_+} A\nabla u\cdot\nabla v\,dZ=-\iint_{{\mathbb R}^n_+} Lu\,v\,dZ=0.$$
Here there is no boundary integral since $v$ vanishes on the boundary of ${{\mathbb R}^n_+}$. It follows that
\begin{equation}\label{e8c}
\|N_{1,\varepsilon}(\nabla u)\|_{L^q}\lesssim\int_{\partial{{\mathbb R}^n_+}}u(x,0)a^*_{nj}(x,0)\partial_jv(x,0)\,dx,
\end{equation}
where the implied constant in \eqref{e8c} is independent of $\varepsilon>0$. Now, we use the fundamental theorem of calculus and the decay of $\nabla v$ at infinity to write \eqref{e8c} as
\begin{equation}\label{e9}
\|N_{1,\varepsilon}(\nabla u)\|_{L^q}\lesssim-\int_{\partial{{\mathbb R}^n_+}}u(x,0)\left(\int_0^\infty \frac{d}{ds}\left(a^*_{nj}(x,s)\partial_jv(x,s)\right)ds\right)dx.
\end{equation}
Recall that $\div(A^*\nabla v)=\div(\vec{h})$ and hence the righthand side of \eqref{e9} equals 
\begin{equation}\label{e10}
\int_{\partial{{\mathbb R}^n_+}}u(x,0)\left(\int_0^\infty \left[\sum_{i<n}\partial_i(a^*_{ij}(x,s)\partial_jv(x,s))-\div\vec{h}(x,s)\right]ds\right)dx.
\end{equation}
We integrate by parts moving $\partial_i$ for $i<n$ onto $u(\cdot,0)$. The integral term containing $\partial_nh_n(x,s)$ does not need to be considered as it equals to zero by the fundamental theorem of calculus since $\vec{h}(\cdot,0)=\vec{0}$ and
$\vec{h}(\cdot,s)\to\vec{0}$ as $s\to\infty$). 

It follows that
\begin{eqnarray}\nonumber
\|N_{1,\varepsilon}(\nabla u)\|_{L^q}&\lesssim& \int_{\partial{{\mathbb R}^n_+}} \nabla_\parallel f(x)\cdot\left(\int_0^\infty \left[\vec{h}_\parallel(x,s)-(A^*\nabla v)_\parallel(x,s)\right]ds\right)dx\\
&=&I+II.\label{e11}
\end{eqnarray}
Here $I$ is the term containing $\vec{h}_\parallel$ and $II$ contains $(A^*\nabla v)_\parallel$. The notation we are using here  is that, for a vector $\vec{w}=(w_1,w_2,\dots,w_n)$, the vector $\vec{w}_\parallel$ denotes the first $n-1$ components of $\vec{w}$, that is $(w_1,w_2,\dots,w_{n-1})$.

As shall see below we do not need worry about term $I$. This is because what we are going to do next is essentially undo the integration by parts we have done above but we swap function $u$ with another better behaving function $\tilde{u}$ with the same boundary data. Doing this we eventually arrive to $\|\tilde N(\nabla\tilde{u})\|_{L^q}$ plus some error terms (solid integrals) that arise from the fact that $u$ and $\tilde{u}$ disagree inside the domain. This explain why we get the same boundary integral as $I$ but with opposite sign as this \lq\lq reverse process" will undo and eliminate all such boundary terms. 

We solve a new auxiliary PDE problem to define $\tilde{u}$. Let $\tilde{u}$ be the solution
of the following boundary value problem for the operator $ L_0$ whose matrix $A_0$ has the block-form
$
A_0=\left[ \begin{array}{c|c}
   A_\parallel & 0 \\
   \midrule
   0 & 1 \\
\end{array}\right] 
$ and

\begin{equation}
{ L}_0 \tilde{u}=\div(A_0\nabla\tilde{u})=0\mbox{ in }\Omega,\qquad \tilde{u}\Big|_{\partial\Omega}=f.\label{e5a}
\end{equation}
Recall that we have assumed that the $L^q$ Regularity problem for the operator 
${L_0}$ is solvable; that is, for a constant $C>0$ independent of $f$, 
$\|\tilde{N}(\nabla \tilde{u})\|_{L^q} \le C\|\nabla_\parallel f\|_{L^q}.$ Then, by \eqref{e34b}, we see that
\begin{equation}
\|\tilde{N}(\nabla \tilde{u})\|_{L^q}+\|S(\nabla\tilde{u})\|_{L^q}\le C\|\nabla_\parallel f\|_{L^q}.
\label{e12}
\end{equation}

We look the term $II$. 
Let
\begin{equation}
\vec{V}(x,t)=-\int_t^\infty (A^*\nabla v)_\parallel(x,s)ds.
\end{equation}
It follows that by the fundamental theorem of calculus
$$II=\int_{\partial{{\mathbb R}^n_+}}\nabla_\parallel u(x,0)\cdot \vec{V}(x,0)dx=\iint_{{{\mathbb R}^n_+}}\partial^2_{tt}\left[\nabla_\parallel \tilde{u}(x,t)\cdot \vec{V}(x,t)\right]t\,dx\,dt,$$
and therefore,
\begin{eqnarray}
II&=&\iint_{{{\mathbb R}^n_+}}\partial^2_{tt}(\nabla_\parallel \tilde{u})\cdot \vec{V}(x,t)t\,dx\,dt+\iint_{{{\mathbb R}^n_+}}\partial_{t}(\nabla_\parallel \tilde{u})\cdot \partial_t(\vec{V}(x,t))t\,dx\,dt\nonumber+\\&&+\iint_{{{\mathbb R}^n_+}}\nabla_\parallel \tilde{u}\cdot \partial^2_{tt}(\vec{V}(x,t))t\,dx\,dt
=II_1+II_2+II_3.
\end{eqnarray}
Here $\tilde{u}$ is same as in \eqref{e5a} (observe that $u$ and $\tilde u$ have the same boundary data).
Since $ \partial_t\vec{V}(x,t)=(A^*\nabla v)_\parallel$ the term $II_2$ is easiest to handle and can be estimated as a product of two square functions
\begin{equation}\label{125}
|II_2|\le \|S(\partial_t\tilde{u})\|_{L^q}\|S(v)\|_{L^{q'}}.
\end{equation}

By our assumption that the $L^{q'}$ Dirichlet problem for the operator ${L}^*$ is solvable, Lemma \ref{l2aa} applies and
provides us with an estimate $\|S(v)\|_{L^{q'}}\le C$. Combining this estimate with \eqref{e12} yields
\begin{equation}
|II_2|\le C\|\nabla_\parallel f\|_{L^q},
\end{equation}
as desired.

Next we look at $II_1$. We integrate by parts moving $\nabla\parallel$ from $\tilde{u}$. This gives us
\begin{equation}\label{eq335}
II_1 =\iint_{{{\mathbb R}^n_+}}\partial^2_{tt}\tilde{u}\cdot\left(\int_{t}^\infty\div_{\parallel}(A^*\nabla v)_{\parallel}ds\right)t\,dx\,dt.
\end{equation}
Using the PDE $v$ satisfies we get that
$$\int_{t}^\infty\div_{\parallel}(A^*\nabla v)_{\parallel}ds=(a_{nj}\partial_jv)(x,t)+\int_{t}^\infty\div\vec{h}\,ds.$$
Using this in \eqref{eq335} we see that
\begin{equation}\label{eq335a}
II_1 =\iint_{{{\mathbb R}^n_+}}(\partial^2_{tt}\tilde{u})(a_{nj}\partial_jv)t\,dx\,dt+\iint_{{{\mathbb R}^n_+}}\partial^2_{tt}\tilde{u}
\cdot\left(\int_{t}^\infty\div\vec{h}\,ds\right)t\,dx\,dt.
\end{equation}
Here the first term enjoys the same estimate as $II_2$, namely \eqref{125}. We work more with the second term which we call $II_{12}$. We integrate by parts in $\partial_t$. 
\begin{eqnarray}\label{eq335b}
II_{12} &=&\iint_{{{\mathbb R}^n_+}}(\partial_{t}\tilde{u})(
\div\vec{h})t\,dx\,dt-\iint_{{{\mathbb R}^n_+}}\partial_{t}\tilde{u}
\cdot\left(\int_{t}^\infty\div\vec{h}\,ds\right)\,dx\,dt\\\nonumber
&=&II_{121}-\iint_{{{\mathbb R}^n_+}}\partial_{t}\tilde{u}
\cdot\left(\int_{t}^\infty\div\vec{h}\,ds\right)\,dx\,dt\\\nonumber
&=&II_{121}+\int_{\partial{\mathbb R^n_+}}\tilde{u}(x,0)\left(\int_0^\infty \div\vec{h}\right)dx+\iint_{{{\mathbb R}^n_+}}\nabla\tilde{u}\cdot\vec{h}\,dx\,dt\\\nonumber
&=&II_{121}-\int_{\partial\mathbb R^n_+}\nabla_\parallel \tilde{u}(x,0)\left(\int_0^\infty \vec{h}_\parallel\right)dx+II_{123}=
II_{121}-I+II_{123}.
\end{eqnarray}
In the second line we have swapped $\partial_t$ and $\partial_{\parallel}$ derivatives  integrating by parts twice. 
This integration yields a boundary term but fortunately this term is precisely as the term $I$ defined by \eqref{e11} but since it comes with opposite sign these two terms cancel out.  We return to the terms $II_{121}$ and $II_{123}$ later.

Next we look at $II_3$. We see that
\begin{equation}
II_3= \iint_{{{\mathbb R}^n_+}}\nabla_\parallel \tilde{u}\cdot \partial_t(A^*\nabla v)_\parallel t\,dx\,dt=
\iint_{{{\mathbb R}^n_+}}\nabla_\parallel \tilde{u}\cdot ((\partial_t A)^*\nabla v)_\parallel t\,dx\,dt
\end{equation}
\begin{equation}\nonumber
+\iint_{{{\mathbb R}^n_+}}\nabla_\parallel \tilde{u}\cdot (A^*\nabla (\partial_t v))_\parallel t\,dx\,dt=II_{31}+II_{32}.
\end{equation}

In order to handle the term $II_{31}$ we will use the fact that the matrix A satisfies the Carleson measure condition \eqref{Carl2}.
The argument uses a stopping time argument that is typical in connection with Carleson measures.

To set this up, let $\mathcal O_j$ denote $\{x \in \partial{{\mathbb R}^n_+}:  N(\nabla\tilde{u})(Q)S(v)(Q) > 2^j\}$ and 
define an enlargement of $\mathcal O_j$ by $\tilde{\mathcal O}_j := \{ M(\chi_{\mathcal O_j}) > 1/2\}$. (Note that  $|\tilde{\mathcal O}_j | \lesssim |\mathcal O_j|$.)
We will break up integrals over ${{\mathbb R}^n_+}$ into regions determined by the sets:
$$F_j = \{X = (y,t) \in {{\mathbb R}^n_+}: |\Delta_{ct}(y) \cap \mathcal O_j| > 1/2, \,\,|\Delta_{ct}(y) \cap \mathcal O_{j+1}| \leq 1/2\},$$
where $c$ depends on the aperture of the cones used to define the nontangential maximal function and square functions. Then, 

\begin{eqnarray}
|II_{31}|&\lesssim& \iint_{{{\mathbb R}^n_+}} |\nabla\tilde{u}||\partial_t A||\nabla v|t dX \le \sum_j \iint_{{{\mathbb R}^n_+} \cap F_j} |\nabla\tilde{u}||\partial_t A||\nabla v|t dX \nonumber\\
&\le& \sum_j \int_{\tilde{\mathcal O}_j \setminus \mathcal O_j} \iint_{\Gamma(Q) \cap F_j} |\nabla\tilde{u}||\partial_t A||\nabla v|t^{2-n} dX dx\nonumber\\
&\le&\sum_j \int_{\tilde{\mathcal O}_j \setminus \mathcal O_j} \left(\iint_{\Gamma(Q)}|\nabla v|^2 |\nabla \tilde{u}|^2 t^{2-n}dX\right)^{1/2}\left(\iint_
{\Gamma(Q) \cap F_j}|\partial_tA|^2|^2t^{2-n}dX\right)^{1/2}dQ\nonumber
\end{eqnarray}
\begin{eqnarray}
&\le&\sum_j \int_{\tilde{\mathcal O}_j \setminus \mathcal O_j} N(\nabla\tilde{u})(Q) S(v)(Q) \left(\iint_
{\Gamma(Q) \cap F_j}|\partial_tA|^2|^2t^{2-n}dX\right)^{1/2} \,dQ\nonumber\\
&\le&\sum_j 2^j \left(\int_{\tilde{\mathcal O}_j}  \iint_
{\Gamma(Q) \cap F_j}|\partial_tA|^2|^2t^{2-n}dX\,dQ\right)^{1/2}  |\tilde{\mathcal O}_j|^{1/2}   \nonumber\\
&\lesssim&\sum_j 2^j  |\mathcal O_j| \,\,\lesssim\,\, \int_{\partial{{\mathbb R}^n_+}} N(\nabla\tilde{u})(Q) S(v)(Q) dQ.\label{Stopping}
\end{eqnarray}

 The penultimate inequality follows from the Carleson measure property of $|\partial_tA|^2| t dX$ as the integration is over the
 Carleson region $\{X=(y,t): \Delta_{ct}(y) \subset \tilde{\mathcal O}_j\}$.
 
 Consequently, by H\"older's inequality,
\begin{equation}
|II_{31}|\lesssim \|S(v)\|_{L^{q'}}\|N(\nabla\tilde{u})\|_{L^q}.
\end{equation}
Hence as above 
\begin{equation}
|II_{31}|\le C\|\nabla_\parallel f\|_{L^q}.
\end{equation}
For the term $ II_{32}$ we separate the parallel and tangential parts of the gradient, to get 
\begin{eqnarray}\nonumber
II_{32}&=&\iint_{{{\mathbb R}^n_+}}\nabla_\parallel \tilde{u}\cdot (A_\parallel^*\nabla_\parallel (\partial_t v))t\,dx\,dt+
\iint_{{{\mathbb R}^n_+}}\nabla_\parallel \tilde{u}\cdot (a_{in}^*\partial^2_{tt} v)_{i<n}t\,dx\,dt\\
&=&-\iint_{{{\mathbb R}^n_+}}\div_\parallel(A_\parallel\nabla \tilde{u})(\partial_tv)tdx\,dt+II_{33}=\iint_{{{\mathbb R}^n_+}}(\partial^2_{tt}\tilde{u})(\partial_tv)tdx\,dt+II_{33}.\nonumber
\end{eqnarray}
Here we have integrated the first term by parts and then used the equation that $\tilde{u}$ satisfies. It follows that in the last expression the first term has square functions bounds identical to \eqref{125}. For $II_{33}$
we write $\partial^2_{tt}v$ as
$$\partial^2_{tt}v=\partial_t\left(\frac{a^*_{nn}}{a^*_{nn}}\partial_t v\right)=\frac1{a^*_{nn}}\partial_t(a^*_{nn}\partial_tv)-\frac{\partial_t a^*_{nn}}{a^*_{nn}}\partial_t v$$
$$=-\frac1{a^*_{nn}}\left[\div_\parallel(A^*_\parallel\nabla_\parallel v)+\sum_{i<n}\left[\partial_i(a^*_{in}\partial_t v)+\partial_t(a^*_{ni}\partial_i v)\right]+\partial_t(a^*_{nn})\partial_t v -\div\vec{h}\right],$$
where the final line follows from the equation that $v$ satisfies. It therefore follows that the term $II_{33}$ can be written as a sum of five terms (which we shall call $II_{331},\dots,II_{335}$).

Terms $II_{331}$ and $II_{332}$ are similar and we deal with then via integration by parts (in $\partial_i$, $i<n$):
\begin{equation}\label{eq133}
|II_{331}|+|II_{332}|\le C\iint_{{\mathbb R}^n_+}|\nabla^2 \tilde{u}||\nabla v|t+C\iint_{{\mathbb R}^n_+}|\nabla A||\nabla\tilde{u}||\nabla v|t.
\end{equation}
For the third term $II_{333}$ we observe that $\partial_t(a^*_{ni}\partial_i v)=\partial_i(a^*_{ni}\partial_t v)+(\partial_ta^*_{ni})\partial_iv-(\partial_ia^*_{ni})\partial_tv$ which implies that it again can be estimated by the right-hand side of \eqref{eq133}. 
The same is true for the term $II_{334}$ which has a bound by the second term on the right-hand side of \eqref{eq133}. It remains to consider the term  $II_{335}$ which is
\begin{equation}\label{eq134}
II_{335}=\sum_{i<n}\iint_{{{\mathbb R}^n_+}}\frac{a_{ni}}{a_{nn}}\partial_i\tilde{u}\,(\div\vec{h})\,t\,dx\,dt.
\end{equation}
Notice the similarity of this term with $II_{121}$, hence the calculation below also applies to it.
We again integrate by parts. Observe we get an extra term when $\partial_t$ derivative falls on $t$. This gives us
\begin{equation}\label{eq135}
|II_{121}|+|II_{335}|\le C\iint_{{\mathbb R}^n_+}|\nabla^2 \tilde{u}||\vec{h}|t+\iint_{{\mathbb R}^n_+}|\nabla A||\nabla\tilde{u}||\vec{h}|t+\sum_{i}\left|\iint_{{\mathbb R}^n_+} \frac{a_{ni}}{a_{nn}}\partial_i\tilde{u}\,h_n\,\,dx\,dt\right|.
\end{equation}

We deal with terms on the right-hand side of \eqref{eq133} and \eqref{eq135} now. The first term of \eqref{eq133}  can be seen to be a product of two square functions and hence by H\"older it has an estimate by $\|S(\nabla\tilde{u})\|_{L^q}\|S(v)\|_{L^{q'}}$. The second term of \eqref{eq133} is similar to the term $II_{31}$ with analogous estimate. It follows that
\begin{eqnarray}\nonumber
|II_{331}|+|II_{332}|+|II_{333}|+|II_{334}|&\le& C(\|S(\nabla\tilde{u})\|_{L^q}+\|N(\nabla\tilde{u})\|_{L^q})\,\|S(v)\|_{L^{q'}}\\\label{eq136}
&\le& C\|\nabla_\parallel f\|_{L^q},
\end{eqnarray}
by using \eqref{e12} and Lemma \ref{l2aa}. The first two terms of \eqref{eq135} have similar estimates, provided we introduce as in \cite{KP2} the operator $\tilde{T}$. Here
$$\tilde{T}(|\vec{h}|)(Q)=\iint_{\Gamma(Q)}|\vec{h}|(Z)\delta(Z)^{1-n}(Z)dZ.$$
The last term of \eqref{eq135} and also the term $II_{123}$ is handled using \eqref{e2}.
Here the presence of $\frac{A_{ni}}{A_{nn}}$ in the integral is harmless as we have flexibility to hide this term into the vector-valued function $\vec\alpha$ in the definition of $\vec{h}$. This gives us
\begin{eqnarray}\nonumber
|II_{121}|+|II_{123}|+|II_{335}|&\le& C(\|S(\nabla\tilde{u})\|_{L^q}+\|N(\nabla\tilde{u})\|_{L^q})\,\|\tilde{T}(|\vec{h}|)\|_{L^{q'}}+C\|\tilde{N}_1(\nabla\tilde{u})\|_{L^q}\\\label{eq136x}
&\le& C\|\nabla_\parallel f\|_{L^q}.
\end{eqnarray}
Here the bound for $\|\tilde{T}(|\vec{h}|)\|_{L^{q'}}$ follows from Lemma 2.13 of \cite{KP2}.
\vglue2mm

In summary, under the assumptions we have made we see that
$$II=\int_{\partial\mathbb R^n_+}\nabla_\parallel u(x,0)\cdot \vec{V}(x,0)dx\le C\|\nabla_\parallel f\|_{L^q}-I.$$

After putting all estimates together (since term $I$ cancels out), we have established the following:
$$ \|\tilde{N}_{1,\varepsilon}(\nabla u)\|_{L^q}\le C\|\nabla_\parallel f\|_{L^q}.$$

\medskip

\noindent {\bf Remark:} The assumption that $L^p$ Regularity problem for the block form operator ${L}_0$ is solvable for some $p>1$ implies solvability of the said Regularity problem for all values of $p\in (1,\infty)$. This 
can be see from implications of section \ref{S3}, namely that 
$(D^*)_p+(R)_1\Longrightarrow (R)_{p'}$ and $(R)_q\Longrightarrow (R)_1$ for all $1<p,q<\infty$. As the 
$(D^*)_p$ is again a block form operator it is solvable for all $1<p<\infty$ and therefore in the block form case we have that $(R)_q\Longrightarrow (R)_1\Longrightarrow (R)_{p'}$ for all $1<p'<\infty$. See  \cite{DK} and \cite{DPP} for more details.\vglue2mm

An argument is required to demonstrate that the control of $\tilde{N}_{1,\varepsilon}(\nabla u)$ of a solution $Lu=0$ implies the control of $\tilde{N}(\nabla u)$ (the $L^2$ averaged version of the non-tangential maximal function). Firstly, as the established estimates are independent of $\varepsilon>0$ we obtain
$$\|\tilde{N_1}(\nabla u)\|_{L^q}=\lim_{\varepsilon\to 0+}\|\tilde{N}_{1,\varepsilon}(\nabla u)\|_{L^q}\le C\|\nabla_\parallel f\|_{L^q}.$$
Secondly, as $\nabla u$ satisfies a reverse H\"older self-improvement inequality as a consequence of Caccioppoli's inequality
$$\left(\fiint_B|\nabla u|^{2+\delta}\right)^{1/(2+\delta)}\lesssim \left(\fiint_{2B}|\nabla u|^{2}\right)^{1/2},$$
for some $\delta>0$ depending on ellipticity constant and all $B$ such that $3B\subset{{\mathbb R}^n_+}$, it also follows (c.f. \cite[Theorem 2.4]{S}) that
$$\left(\fiint_B|\nabla u|^2\right)^{1/2}\lesssim \left(\fiint_{2B}|\nabla u|\right),$$
which implies a bound of $\tilde{N}(\nabla u)(\cdot)$ defined using cones $\Gamma_a(\cdot)$ of some aperture $a>0$ by $\tilde{N}_1(\nabla u)(\cdot)$ defined using cones $\Gamma_b(\cdot)$ of some slightly larger aperture $b>a$. Hence $ \|\tilde{N}(\nabla u)\|_{L^q}\le C\|\nabla_\parallel f\|_{L^q}$ must hold.
This completes the proof of Theorem \ref{t1}. \qed\medskip

It remains to prove Theorem \ref{tblock}. \medskip

\noindent {\it Proof of Theorem \ref{tblock}.} Consider therefore $A_\parallel$ as in Theorem \ref{tblock}  and denote by $ L_0$ the operator
\begin{equation}\label{e0ax}
{ L}_0 u= \mbox{\rm div}_\parallel(A_\parallel \nabla_\parallel u)+u_{tt}.
\end{equation}

For each $k=2,3,4,\dots$ let $ L_k$ be a related rescaled operator in $t$-variable defined as follows:
\begin{equation}\label{e0axs}
{ L}_k u= \mbox{\rm div}_\parallel(A^k_\parallel \nabla_\parallel u)+u_{tt},
\end{equation}
where 
\begin{equation}\label{ss1}
A^k_\parallel(x,t)=A_\parallel(x,kt),\qquad\mbox{for all }x\in\mathbb R^{n-1}\mbox{ and }t>0. 
\end{equation}

We claim that for each $k=2,3,\dots$ the $L^q$ Regularity problem for $L_0$ in $\mathbb R^n_+$  is solvable if and only if the $L^q$ Regularity problem for $ L_k$ in $\mathbb R^n_+$  is solvable.

This can be see as follows. Using the mean value theorem, the coefficients $A^k_\parallel$ can be viewed as Carleson perturbations of coefficients of $ L_0$ which are $A_\parallel$. That is, similar to \eqref{CCpert}, we have that
\begin{equation}\label{CCpertX}
\delta(X)^{-1}\left[\sup_{Y\in B(X,\delta(X)/2)}|A_\parallel(Y)-{A^k_\parallel}(Y)|\right]^2\mbox{ is a Carleson measure.}
\end{equation}
Thus, if the $L^q$ Regularity problem for $ L_0$ in $\mathbb R^n_+$ is solvable, then so is the 
$L^{\tilde{q}}$ Regularity problem for $ L_k$ in $\mathbb R^n_+$ for some $\tilde{q}>1$ by Theorem \ref{T:perturbation2}.
But for these block form operators, solvability of the Regularity problem for one value $\tilde{q}>1$ implies solvability for all values (because the Dirichlet problem for the adjoint is solvable for all $1<q'<\infty$). Therefore we can deduce that the $L^q$ Regularity problem for $L_k$ in $\mathbb R^n_+$ is solvable. The reverse implication has a similar proof. 

Next, we consider what we can say about the Carleson condition for the coefficients $A^k_\parallel$. We want to look at 
\begin{equation}\label{e5.55}
d\mu^k(x,t)=|\nabla_x A^k_\parallel (x,t)|^2t\,dx\,dt.
\end{equation}
Notice that the gradient is only taken in $x$ variable, not in $t$, so we are not examining the same (full) Carleson measure property of the coefficients.
Given that \eqref{Carl2m} holds,  it follows that for 
$$d\mu^0(x,t)=|\nabla_x A_\parallel (x,t)|^2t\,dx\,dt,$$
we have that
\begin{equation}\label{ss2}
\|\mu^0\|_{Carl}\le \|\mu\|_{Carl}\qquad\mbox{and }\qquad |\nabla_x A_\parallel (x,t)|\le \frac{\|\mu\|^{1/2}_{Carl}}t.
\end{equation}
Let $\Delta\subset\mathbb R^{n-1}$ be a boundary ball of radius $r$. Let $T(\Delta)$ be the usual Carleson region associated with $\Delta$.

 To estimate the Carleson norm of $\mu^k$ in the region $T(\Delta)\cap \{X:\delta(X)<r/k\}$, a change of variables $(x,t)\mapsto (x,kt)$  together with the first
  the Carleson norm property in
  \eqref{ss2} gives an upper bound of $1/k^2$.
 In the region $T(\Delta)\cap \{X:\delta(X)\ge r/k\}$, 
we use the second estimate in \eqref{ss2} and altogether this gives:
\begin{equation}\label{ss3}
\|\mu^k\|_{Carl}\le \|\mu\|_{Carl}\frac{1+C(n)\log k}{k^2},\qquad\mbox{for some }C(n)>0.
\end{equation}
It follows that by choosing $k$ large enough we can make the Carleson norm of $\mu^k$ as small as we wish.
This observation will be crucial for what follows.

From now on let $B_\parallel=A^k_\parallel$ for some large fixed $k$ which will be determined later. Let 
\begin{equation}\label{e0axx}
{L} u= \mbox{\rm div}_\parallel(B_\parallel \nabla_\parallel u)+u_{tt},
\end{equation}
and we consider the Regularity problem for this operator on $\Omega=\mathbb R^n_+$. Our objective now is to solve the $L^q$ Regularity problem for $L$ for some $q>1$, thus proving  Theorem \ref{tblock}.

Suppose that $Lu=0$ and that $u\big|_{\partial\Omega}=f$ for some $f$ with $\nabla_x f\in L^q$. Let us recall \eqref{eqDer} but now applied to the block-form case. It follows that we have the following for $v_j=
\partial_j u$:
\begin{eqnarray}\label{system}
{ L}v_m&=&\sum_{i,j=1}^{n-1}\partial_i((\partial_m b_{ij})v_j)\quad\mbox{in }\Omega,\quad m=1,2,\dots,n-1,\\
v_m\Big|_{\partial\Omega}&=&\partial_m f.\nonumber
\end{eqnarray}
Observe that only $v_1,\dots, v_{n-1}$ appears in these equations and hence \eqref{system} is a {\it weakly coupled} fully determined system of $n-1$ equations for the unknown vector valued function $V=(v_1,v_2\dots,v_{n-1})$
with boundary datum $V\big|_{\partial\Omega}=\nabla_xf\in L^p$. We call this system {\it weakly coupled} because each $\partial_mb_{ij}$ appearing on the righthand side has small Carleson measure norm, which follows from \eqref{ss3} since $k$ will be chosen to be (sufficiently) large. 

In particular, Lemma \ref{ll1} applies here but because $v_n$ does not appear in \eqref{system} we will only have $N(\nabla_T u)=N(V)$ on the right-hand side. That is:
\begin{eqnarray}
\label{e8ss} && \int_{\partial{\mathbb
R}^n_+}S(V)^2\,dx \le C\int_{\partial{\mathbb
R}^n_+}|V|^2 dx+\|\mu^k\|_{Carl}\int_{\partial{\mathbb
R}^n_+}N(V)^2 dx,
\end{eqnarray}
where $\|\mu^k\|_{Carl}$ is the  (partial) Carleson norm  of coefficients of the operator $L^k=L$ as defined by \eqref{e5.55}. This requires revisiting some of the arguments in the proof of Lemma \ref{ll1} by checking that $\partial_t B$ only appears for the coefficients of the last row and column on the matrix $B=A^k_\parallel$. Hence those are all equal to zero and do not show up in the formula \eqref{e8ss}.\medskip

It remains to establish nontangential estimates of $N(V)$ since we would like to move such terms from the righthand side of \eqref{e8ss}. This follows from the following lemma from \cite{DHP} which we address in the appendix:

\begin{lemma}\label{S3:C7-alt1} Under the above assumptions on $L$, for sufficiently large $k$ we have 
that for any $p>1$ and $a>0$ there exists a 
constant $C>0$ such that 
\begin{equation}\label{S3:C7:E00oo}
\|{N}(V)\|_{L^{p}({\mathbb R}^{n-1})}\le C\|S(V)\|_{L^{p}({\mathbb R}^{n-1})}+C\|V\|_{L^{p}({\mathbb R}^{n-1})}.
\end{equation}
\end{lemma}

We can now combine Lemma \ref{S3:C7-alt1} with estimate \eqref{e8ss}. It follows as before that
\begin{eqnarray}\label{sss33}
\|{N}(V)\|_{L^{2}({\mathbb R}^{n-1})}&\le& C\|S(V)\|_{L^{2}({\mathbb R}^{n-1})}+C\|\nabla_x f\|_{L^{2}({\mathbb R}^{n-1})}\\\nonumber
&\le& C\|\nabla_x f\|_{L^{2}} + C\|\mu\|^k_{Carl} \|{N}(\vec\eta)\|_{L^{2}({\mathbb R}^{n-1})}.
\end{eqnarray}
For $k$ chosen so large that the constant $C\|\mu\|^k_{Carl}<1/2$ we then obtain
\begin{equation}\label{S3:C7:E00ooss}
\|{N}(\nabla_T u)\|_{L^{2}({\mathbb R}^{n-1})}=\|{N}(V)\|_{L^{2}({\mathbb R}^{n-1})}\le 2C\|\nabla_x f\|_{L^{2}({\mathbb R}^{n-1})}.
\end{equation}
This estimate nearly establishes that the $L^2$ Regularity problem for $L=L_k$ is solvable, the only missing part is the corresponding nontangential estimate for $v_n=\partial_tu$. \medskip

We start with the square function estimate for $v_n$. Since we already have estimates for $S(V)$ the only remaining term that needs an estimate is
$\iint_{\mathbb R^n_{+}}|\partial_{tt}u|^2tdt\,dx$. Since $Lu=0$ this equation shows that
\begin{equation}
\iint_{\mathbb R^n_{+}}|\partial_{tt}u|^2tdt\,dx= \iint_{\mathbb R^n_{+}}\sum_{i,j,s,r<n}\partial_i(b_{ij}\partial_ju)\partial_s(b_{sr}\partial_ru)t\,dt\,dx
\end{equation}
$$\le C\|S(\nabla_Tu)\|^2_{L^2(\mathbb R^{n-1})}+C\iint_{\mathbb R^n_{+}}|\nabla_x B_\parallel|^2|\nabla_Tu|^2t\,dt\,dx$$
$$\le C\|S(\nabla_Tu)\|^2_{L^2(\mathbb R^{n-1})}+C\|\mu^k\|_{Carl}\|N(\nabla_Tu)\|^2_{L^2(\mathbb R^{n-1})}.$$

It is again possible to establish an analogue of Lemma \ref{S3:C7-alt1} for $v_n$ (c.f. \cite{DHP}).
\begin{equation}\label{S3:C7:E00oozz}
\|N(v_n)\|_{L^p(\mathbb R^{n-1})}\le C\|S(v_n)\|_{L^p(\mathbb R^{n-1})}+C\||\nabla B|^2t\|_{Carl}\|N(\nabla_T u)\|_{L^p(\mathbb R^{n-1})}. 
\end{equation}
Here the Carleson norm in the above estimate might not be small as it involves the $\partial_t$ derivative of $B$. That is not an issue however since we already have bounds of $\|N(\nabla_T u)\|_{L^2(\mathbb R^{n-1})}$ and $\|S(\nabla_T u)\|_{L^2(\mathbb R^{n-1})}$ by $C\|\nabla_x f\|_{L^{2}}$ from \eqref{S3:C7:E00ooss}.
Hence also $|N(v_n)\|_{L^2(\mathbb R^{n-1})}\lesssim \|\nabla_x f\|_{L^{2}}$ and therefore the Regularity problem in $L^2$ for $L=L_k$ is solvable on $\mathbb R^{n}_+$.
As this also implies solvability for $L_0$, the argument is complete. \qed


\section{The Neumann problem}

We first consider the Neumann problem in dimension $n=2$ with the large Carleson condition imposed on the coefficients of our matrix. 
The solvability of the Neumann problem
can be reduced to solvability of the Regularity problem using an observation in \cite{KR}; namely, if $u$
solves $ Lu=\div(A\nabla u)=0$ in a Lipschitz domain $\Omega$ then $\tilde{u}$ uniquely (modulo constants) defined via
\begin{equation}\label{bmatrix}
\begin{bmatrix}
0 & -1\\ 1 & 0
\end{bmatrix}\nabla{\tilde{u}}=A\nabla u
\end{equation}
solves the equation $\tilde{ L}u=\div(\tilde{A}\nabla u)=0$ with $\tilde{A}=A^t/\det{A}$ and the tangential derivative of $u$ on $\partial\Omega$ is the co-normal derivative of $\tilde{u}$ on $\partial\Omega$ and vice-versa. 

If $A$ satisfies the Carleson condition \eqref{carlMMM} then so does $A^t/\det{A}$ (with a possibly larger constant) and hence the $L^p$ Neumann problem for a given matrix $A$ is solvable in the same range $1<p<p_{max}$ for which the $L^p$ Regularity problem for the matrix $A^t/\det{A}$ is solvable. The range of solvability for the operator with matrix $A^t/\det{A}$ is determined by the range of solvability of 
the Dirichlet problem for its adjoint operator, which has matrix $A/\det{A}$, reducing the second claim of Theorem \ref{RNduality} to the first (about the Regularity problem) and we have already shown that.

In summary, in two dimensions, the solvability of the Neumann problem can be deduced from solvability of the Regularity problem for a related operator whose coefficients also satisfy the Carleson condition.

The large Carleson condition case for the Neumann problem is open in dimensions larger than two
even in smooth domains.\medskip

Next, we consider the Neumann problem under the small Carleson condition, where the results from \cite{DPR} apply and give us 
solvability of $(N)_p$ for all $1<p<\infty$ in all dimensions.\medskip

For simplicity we only outline here the case $p=2$, that is the $(N)_2$ Neumann problem. The full proof for all $p$ can be found in \cite{DPR} but its main idea is already contained in the $p=2$ case presented below.

Again, our standard reductions apply and we can focus on the case of $L$ on $\mathbb R^n_+$ satisfying the small Carleson condition \eqref{CCalt}. By the considerations of the previous section, we know that if 
the Carleson norm of coefficients is small enough then
$(R)_2$ BVP is solvable. This gives us the estimate
\begin{equation}\label{e7b}
\|N(\nabla u)\|_{L^2(\partial {\mathbb R}^n_+)} \lesssim
\|\nabla_T f\|_{L^2(\partial {\mathbb R}^n_+)}.
\end{equation}
We want to control $\|N(\nabla u)\|_{L^2(\partial {\mathbb
R}^n_+)}$ in terms of the co-normal derivative $A\nabla u\cdot
\nu\big|_{\partial\Omega}$. However, since in our case
$\Omega={\mathbb R}^n_+$ this is just
$$H\big|_{\{t=0\}},\qquad\text{where:}\qquad H=a_{ni}\partial_i u.$$
In the light of (\ref{e7b}) it suffices to prove these two inequalities:
\begin{equation}\label{e7c}
\|\nabla_T f\|_{L^2(\partial {\mathbb R}^n_+)}\lesssim
\|S(H)\|_{L^2(\partial {\mathbb R}^n_+)},
\end{equation}
and
\begin{equation}\label{e7d}
\|S(H)\|_{L^2(\partial {\mathbb R}^n_+)}\lesssim
\|H\|_{L^2(\partial {\mathbb R}^n_+)},
\end{equation}
for sufficiently small Carleson norm of coefficients. Then \eqref{e7b}-\eqref{e7d} together imply that $(N)_2$ is solvable.
\vglue2mm

We start with the estimate (\ref{e7c}). Denote again by
$v_k=\partial_ku$ for $k=1,2,\dots,n$. For each $k\le n-1$ we have
\begin{eqnarray}
\nonumber \int_{{\mathbb R}^{n-1}}|v_k(x,0)|^2dx&=&-\iint_{{\mathbb
R}^{n}_+}\partial_n(|v_k|^2)(X)\,dX\\
\label{NPe2} &=&-2\iint_{{\mathbb R}^{n}_+}v_k(\partial_nv_k)\,dX.
\end{eqnarray}

Since $\partial_nv_k=\partial_kv_n$ we have that this equals to

\begin{eqnarray}
&&-2\iint_{{\mathbb R}^{n}_+}v_k(\partial_kv_n)\,dX\nonumber\\
&=&-2\iint_{{\mathbb
R}^{n}_+}v_k\partial_k\left(\frac{a_{ni}}{a_{nn}}v_i\right)\xi\,dX+2\sum_{i<n}\iint_{{\mathbb
R}^{n}_+}v_k\partial_k\left(\frac{a_{ni}}{a_{nn}}v_i\right)\,dX.\label{NPe3}
\end{eqnarray}

The second term of (\ref{NPe3}) can be further written as
\begin{eqnarray}
2\sum_{i<n}\iint_{{\mathbb
R}^{n}_+}v_kv_i\partial_k\left(\frac{a_{ni}}{a_{nn}}\right)\,dX+\sum_{i<n}\iint_{{\mathbb
R}^{n}_+}\partial_i(|v_k|^{2})\frac{a_{ni}}{a_{nn}}\,dX.\label{NPe4}
\end{eqnarray}
We introduce $(\partial_n t)$ into the both terms of (\ref{NPe4})
and integrate by parts. This gives us:

\begin{eqnarray}
&&-\sum_{i<n}\left[2\iint_{{\mathbb
R}^{n}_+}\partial_n\left(v_kv_i\partial_k\left(\frac{a_{ni}}{a_{nn}}\right)\right)
t\,dX+\iint_{{\mathbb
R}^{n}_+}\partial_n\left(\partial_i(|v_k|^{2})\frac{a_{ni}}{a_{nn}}\right)
t\,dX\right]\nonumber\\
&=&-\sum_{i<n}\left[2\iint_{{\mathbb
R}^{n}_+}(\partial_nv_k)v_i\partial_k\left(\frac{a_{ni}}{a_{nn}}\right)
t\,dX+2\iint_{{\mathbb
R}^{n}_+}v_k(\partial_nv_i)\partial_k\left(\frac{a_{ni}}{a_{nn}}\right)
t\,dX                 \right. \nonumber\\
&&\hskip2mm +\left. \iint_{{\mathbb
R}^{n}_+}\partial_i(|v_k|^{2})\partial_n\left(\frac{a_{ni}}{a_{nn}}\right)
t\,dX\right]\label{NPe5}\\
&&-\sum_{i<n}\left[2\iint_{{\mathbb
R}^{n}_+}v_kv_i\partial_n\partial_k\left(\frac{a_{ni}}{a_{nn}}\right)
t\,dX+\iint_{{\mathbb
R}^{n}_+}\partial_n\partial_i(|v_k|^{2})\frac{a_{ni}}{a_{nn}}
t\,dX\right]\nonumber.
\end{eqnarray}

The last two terms can be integrated by parts one more time as we
switch the order of derivatives. This gives
\begin{eqnarray}
&&\sum_{i<n}\left[2\iint_{{\mathbb
R}^{n}_+}\partial_k\left(v_kv_i\right)\partial_n\left(\frac{a_{ni}}{a_{nn}}\right)
t\,dX+\iint_{{\mathbb
R}^{n}_+}\partial_i\left(\frac{a_{ni}}{a_{nn}}\right)\partial_n
(|v_k|^2) t\,dX\right].\label{NPe6}
\end{eqnarray}

The first three terms on the righthand side of (\ref{NPe5}) and
both terms of (\ref{NPe6}) can be bounded from above by

\begin{eqnarray}
&&\hskip-10mm C\iint_{{\mathbb R}^{n}_+}|\nabla u||\nabla^2 u||\nabla
A|t\,dX\label{NPe10}\\
&\le& \left(\iint_{{\mathbb R}^{n}_+}|\nabla^2
u|^2t\,dX\right)^{1/2}\left(\iint_{{\mathbb R}^{n}_+}|\nabla
u|^{2}|\nabla A|^2t\,dX\right)^{1/2}\nonumber\\&+&\|\mu\|_{Carl}\|S(\nabla
u)\|_{L^2({\partial\mathbb R}^{n}_+)}\|N(\nabla u)\|_{L^2({\partial\mathbb
R}^{n}_+)}.\nonumber
\end{eqnarray}
The first term of (\ref{NPe3}) can be written as

\begin{eqnarray} \nonumber &&-2\iint_{{\mathbb
R}^{n}_+}v_k\partial_k\left(\frac{H}{a_{nn}}\right)\,dX=-2\iint_{{\mathbb
R}^{n}_+}v_k\partial_k\left(\frac{H}{a_{nn}}\right)(\partial_n
t)\,dX\\
&=& 2\iint_{{\mathbb
R}^{n}_+}(\partial_nv_k)\partial_k\left(\frac{H}{a_{nn}}\right)
t\,dX+ \label{NPe7} 2\iint_{{\mathbb
R}^{n}_+}v_k\partial_n\partial_k\left(\frac{H}{a_{nn}}\right)
t\,dX,
\end{eqnarray}
where the last term further yields:
\begin{eqnarray}
&&\label{NPe8} -2\iint_{{\mathbb
R}^{n}_+}(\partial_kv_k)\partial_n\left(\frac{H}{a_{nn}}\right)
t\,dX.
\end{eqnarray}

If the derivative in the first two terms on the righthand side of
(\ref{NPe7}) and (\ref{NPe8}) falls on the coefficients of the
matrix $A$ we obtain terms we have already bounded above as in
\ref{NPe10}. If the derivative falls on $H$ the first terms on
the righthand side of both (\ref{NPe7}) and (\ref{NPe8}) are
bounded by
\begin{eqnarray}
\nonumber C\iint_{{\mathbb R}^{n}_+}|\nabla v_k||\nabla H|
t\,dX&\le& C\left(\iint_{{\mathbb R}^{n}_+}|\nabla v_k|^2
t\,dX\right)^{1/2} \left(\iint_{{\mathbb R}^{n}_+}|\nabla H|^2
t\,dX\right)^{1/2}\\\nonumber &\le&C\|S(v_k)\|_{L^2(\partial{\mathbb
R}^{n}_+)}\|S(H)\|_{L^2({\partial\mathbb R}^{n}_+)}.
\end{eqnarray}

Summing over all $k<n$ this yields a global estimate
\begin{eqnarray}
\nonumber \int_{\partial{\mathbb R}^{n}_+}|\nabla_T u|^2dx&\le&
C\|S(v_k)\|_{L^2(\partial{\mathbb R}^{n}_+)}
\|S(H)\|_{L^2(\partial{\mathbb R}^{n}_+)}\\&+& \|\mu\|_{Carl}\|S(\nabla
u)\|_{L^2(\partial{\mathbb R}^{n}_+)}\|N(\nabla
u)\|_{L^p(\partial{\mathbb R}^{n}_+)}.\nonumber
\end{eqnarray}

From this by (\ref{e7b}) and using Lemma \ref{ll3} we get that for
all sufficiently small norm of $\|\mu\|_{Carl}$ the desired estimate
(\ref{e7c}) holds.\vglue2mm

We now look at the estimate (\ref{e7d}). By (\ref{e7c}) we know that for sufficiently large $K>0$ the inequality
\begin{eqnarray} &&
\int_{\partial{\mathbb R}^{n}_+} N^{2}(\nabla u)\,dx+
\iint_{{\mathbb R}^{n}_+}|\nabla H|^2t\,dX\label{NPe16a} \le K
\iint_{{\mathbb R}^{n}_+}|\nabla H|^2t\,dX
\end{eqnarray}
holds. Clearly,
\begin{eqnarray} \nonumber
K\iint_{{\mathbb R}^{n}_+}|\nabla
H|^2t\,dX\approx\iint_{\R^n_+}b_{ij}(\partial_iH)(\partial_jH)
t\,dX,
\end{eqnarray}
for some matrix $B$ satisfying the ellipticity condition to be
specified later (but we now set $b_{nn}=1$). We use \eqref{NPe16a} and apply Lemma \ref{lgeneral} to $H$. This gives:

\begin{eqnarray} &&\hskip10mm
\int_{\partial{\mathbb R}^{n}_+} N^{2}(\nabla u)\,dx+
\iint_{{\mathbb R}^{n}_+}|\nabla H|^2t\,dX\label{NPe16ax} \\&\le&\nonumber C\int_{\partial{\mathbb R}^n_+}|H|^2dX+C\|\mu\|_{Carl}\int_{\partial{\mathbb
R}^n_+}N(H)^2 dx-C\iint_{{\mathbb R}^n_+}\frac1{a_{nn}}(\tilde LH)Ht\,dX.
\end{eqnarray}
By assuming that $C\|\mu\|_{Carl}<\frac12$ we then get that

\begin{eqnarray} &&\nonumber
\frac12\int_{\partial\R^n_+} N^{2}(\nabla u)\,dx+
\iint_{\R^n_+}|\nabla H|^2t\,dX\\\label{NPe24} &\le&
C\int_{\partial\R^n_+}|H|^{2}\,d\sigma-
C\iint_{\R^n_+}H(\widetilde{L}H) t\,dX.
\end{eqnarray}
Here $\widetilde{L}H=\mbox{ div}(B\nabla H)$. Clearly,
(\ref{NPe24}) implies the desired estimate (\ref{e7d}) modulo the last
extra term which we shall consider now.\vglue2mm

Using the summation convention, i.e., only writing sums
whenever the sum is not taken over all indices, we have 
\begin{eqnarray} &&\nonumber \widetilde{L}H=\partial_i(b_{ij}\partial_j
H)=\sum_{j<n} \partial_i(b_{ij}\partial_j(a_{nk} \partial_k u))+
\partial_i(b_{in}\partial_n(a_{nk} \partial_k u)).
\end{eqnarray}
Since $Lu=0$ we know that $\partial_n(a_{nk} \partial_k
u)=-\sum_{j<n}\partial_j(a_{jk} \partial_k u)$. Hence
\begin{eqnarray} &&\nonumber \widetilde{L}H=\partial_i(b_{ij}\partial_j
H)=\sum_{j<n} [\partial_i(b_{ij}\partial_j(a_{nk} \partial_k u))-
\partial_i(b_{in} \partial_j(a_{jk} \partial_k u))].
\end{eqnarray}
We also swap the role of $i$ and $k$ in the second term. From this
\begin{eqnarray} &&\label{NPe25}\widetilde{L}H=\partial_i(b_{ij}\partial_j
H)=\sum_{j<n} [\partial_i(b_{ij}\partial_j(a_{nk} \partial_k u))-
\partial_k(b_{kn} \partial_j(a_{ji} \partial_i u))].
\end{eqnarray}
We choose $b_{ij}=a_{ji}/a_{nn}$. Notice that this guarantees that
$b_{nn}=1$ as desired. We first look at the terms in (\ref{NPe25}) where all three derivatives
fall on $u$. We claim that such terms all cancel out since they are:
\begin{eqnarray} &&\label{NPe26}\\&&\nonumber\sum_{j<n} [b_{ij}a_{nk} (\partial_i\partial_j\partial_k u)-
b_{kn}a_{ji} (\partial_i\partial_j\partial_k
u)]=\sum_{j<n}a_{nn}^{-1}(a_{ji}a_{nk}-a_{nk}a_{ji})\partial_i\partial_j\partial_k
u=0.
\end{eqnarray}

It follows that the last term of (\ref{NPe24}) can be written as
\begin{eqnarray} &&\label{NPe27}
\iint_{\R^n_+}\sum_{j<n}[b_{ij}(\partial_i\partial_j
a_{ij})(\partial_k u)-b_{kn}(\partial_k\partial_j
a_{ji})(\partial_i u)]H t dX
\end{eqnarray}
plus terms with bound:
\begin{eqnarray} &&\label{NPe28}
\iint_{\R^n_+}|\nabla u|[|\nabla u||\nabla A||\nabla B|+|\nabla^2
u||\nabla A||B|+|\nabla^2 u||\nabla B||A|]t\,dX.
\end{eqnarray}
Terms (\ref{NPe27}) have two derivatives on coefficients $a_{ij}$
however one is $\partial_j$, $j<n$. We therefore integrate by
parts in $\partial_j$. This yields additional terms, but all are
of the form that can be bounded by (\ref{NPe28}). Clearly the terms in (\ref{NPe28}) are
bounded by $\|\mu\|_{Carl}[\|N(\nabla u)\|^2_{L^2}+\|N(\nabla
u)\|_{L^2}\|S(\nabla u)\|_{L^2}]$. Hence for sufficiently
small Carleson norm of $\mu$ these terms can absorbed in (\ref{NPe24})
within the term $\frac12\int_{\partial{\mathbb R^n_+}} N^{2}(\nabla u)\,dx$ on the lefthand side.
This yields the desired estimate (\ref{e7d}) and concludes our proof.\qed

\section{Weaker Geometric Conditions on the boundary of the domain}
\setcounter{equation}{0}

We conclude this survey by pointing out a few of the exciting developments in elliptic theory in the 
setting of rough domains. The domains under investigation satisfy a variety of geometric conditions strictly weaker than the Lipschitz condition, 
and the question is how much of the existing theory of harmonic measure, or even beyond to the case of elliptic measure, as well
as solvability of various boundary value problems, extends to such domains.

We do not pretend to give a thorough overview in this survey paper of this decades long developing, but now rather fast moving, area. Rather, we highlight some of the striking developments and provide a few
references for further reading. In particular, the recent paper \cite{MPT} deals with uniformly rectifiable domains satisfying a corkscrew condition and gives a very nice historical overview with a lot of the important references, especially regarding the Regularity problem. The paper \cite{HMMTZ}, whose main results are
 mentioned below, also has a fairly comprehensive introduction to developments connecting regularity of the elliptic measure (the $A_\infty$ condition) to geometric 
 properties of the boundary.

To motivate the definitions coming up, we note first a hierarchy of domains that are defined by some geometric conditions that are natural in the context of
 the harmonic or elliptic theory. 
\medskip

A {\it uniform domain} satisfies the interior corkscrew condition and the Harnack chain condition.

An {\it nontangentially accessible (NTA) domain} is a uniform domain that satisfies the exterior corkscrew condition. (\cite{JKnta})

A {\it chord arc domain} is an NTA domain whose boundary is Ahlfors-David (AD) regular.

A uniform domain whose boundary is uniformly rectifiable (\cite{DS}) satisfies an exterior corkscrew condition and is therefore chord arc. (See Theorem 
\ref{rectchordarc} below.)
\medskip

The investigation into boundedness of singular integrals, properties of harmonic measure, and regularity of solutions to boundary value problems started decades ago and is converging towards a rather complete theory for elliptic operators of the form $ L=\div(A\nabla)$ with coefficients satisfying the Carleson condition 
\eqref{CCalt}.  This condition on coefficients can also be replaced by the hypotheses that
\begin{enumerate}
\item $A$ is Lipschitz and $\delta(X)\nabla A(X)$ is bounded, and 
\item $\delta(X) |\nabla A(X)|^2 dX$ is a Carleson measure.
\end{enumerate}

As we have seen, there are major differences in the 
nature of the results one might expect under the different assumptions: namely, that Carleson norm of the coefficients is
arbitrarily large, and that the Carleson norm has the vanishing property or is sufficiently small. In the latter case, solvability of the boundary value problem in
all $L^p$ spaces, $1<p<\infty$ is expected, whereas under the former assumption, only in some $L^p$ space.

\subsection{Definitions} 

We now define some of the geometric properties referred to above.

\smallskip

A closed set $E \subset {\mathbb R}^{n}$ is $n-1$-dimensional AD-regular  if
there is some uniform constant $C$ such that for $\sigma=H^{n-1}$ (the $n-1$ dimensional Hausdorff measure)
\begin{equation} \label{eq1.ADR}
\frac1C\, r^{n-1} \leq \sigma(E\cap B(Q,r)) \leq C\, r^{n-1},\,\,\,\forall r\in(0,R_0),Q \in E,
\end{equation}
where $R_0$ is the diameter
of $E$ (which may be infinite).

As we shall only be discussing $n-1$-dimensional regularity, we subsequently drop the reference to $n-1$.

\smallskip

A measure $\mu$ is {\it uniformly rectifiable} if it is AD regular and there exist constants $\theta,M>0$ such that for each  $x$ in the support of $\mu$ and each $r$ smaller than the diameter of the support of $\mu$, there is a Lipschitz mapping $g$, with Lipschitz constant less than $M$, from the ball $B(0,R)$ to 
$\mathbb R^{n-1}$ satisfying the bound $\mu\big(B(x,r)\cap g(B(0,r))\big)\geq\theta r^{n-1}.$

\smallskip

A set $E \subset {\mathbb R}^{n}$ is uniformly rectifiable if $\sigma=H^{n-1}$ is uniformly rectifiable.

\smallskip

A domain $\Omega\subset {\mathbb R}^{n}$
satisfies the {\it corkscrew condition} if for some uniform constant $c>0$ and
for every surface ball $\Delta:=\Delta(Q,r),$ with $Q\in \partial\Omega$ and
$0<r<\diam(\partial\Omega)$, there is a ball
$B(x_\Delta,cr)\subset B(Q,r)\cap\Omega$.  The point $x_\Delta\subset \Omega$ is called
a {\it corkscrew point relative to} $\Delta$ (or, relative to $B$).

\smallskip

A domain
$\Omega$ satisfies the {\it Harnack Chain condition} if there is a uniform constant $C$ such that
for every $\rho >0,\, \Lambda\geq 1$, and every pair of points
$x,x' \in \Omega$ with $\delta(x),\,\delta(x') \geq\rho$ and $|x-x'|<\Lambda\,\rho$, there is a chain of
open balls
$B_1,\dots,B_N \subset \Omega$, $N\leq C(\Lambda)$,
with $x\in B_1,\, x'\in B_N,$ $B_k\cap B_{k+1}\neq \emptyset$
and $C^{-1}\diam (B_k) \leq \dist (B_k,\partial\Omega)\leq C\diam (B_k).$  The chain of balls is called
a {\it Harnack Chain}.

\smallskip

As we stated at the beginning of this section, a domain $\Omega$ that satisfies both the corkscrew and Harnack Chain conditions is a
 {\it uniform domain}, and is also called a {\it 1-sided NTA domain}; the class of NTA and chord arc domains have also been defined above.

\smallskip

\subsection{Further results on rough domains}

We now give a sampling of some of the recent work on boundary value problems in the setting of rough domains.  We begin by noting that since hypotheses (1) and (2) are preserved on Lipschitz subdomains, the Dirichlet problem solved in \cite{KP} is also known to be solvable on chord arc domains.
This was pointed out in \cite{HMMTZ}, where
the authors went on to establish, for uniform domains whose boundary is AD regular,  the
equivalence between uniform rectifiability of the boundary and regularity of elliptic measure for operators satisfying conditions (1) and (2) above. This 
landmark result had been previously established first for the Laplacian and then under the assumption of smallness of the Carleson norm of the coefficients.
\medskip

\begin{theorem}[\cite{HMMTZ}]\label{rectchordarc}
	Let $\Omega\subset \mathbb R^n$, $n\ge 3$, be a uniform domain with AD regular boundary and set $\sigma=H^{n-1}|_{\partial \Omega}$. Let $A$ be a (not necessarily symmetric) uniformly elliptic matrix on $\Omega$ satisfying assumptions (1) and (2) . Then the following are equivalent:
\begin{itemize}
		\item\label{1-thm-main} The elliptic measure $\omega_L$ associated with the operator $ L=\mbox{div} (A\nabla)$ is of class $A_\infty$ with respect to the surface measure. 
		\item\label{3-thm-main} $\partial \Omega$ is uniformly rectifiable.
		\item\label{2-thm-main} $\Omega$ is a chord-arc domain.
		\end{itemize} 
\end{theorem}

Additionally, the solvability of the Dirichlet problem for operators satisfying (1) and (2) holds 
 in domains even weaker than chord arc. Although not explicit in \cite{AHMMT}, the authors
 of \cite{MPT} note that the result stated in \cite{AHMMT} for harmonic functions holds for this class of operators on the domains considered
 there and defined by the  {\it Interior Big Pieces of Chord Arc Domains} 
(IBPCAD) condition. See \cite{AHMMT} for the definition.
 
\medskip

Recently, and building on a delicate construction in \cite{MT} for harmonic functions, more progress has been
made on the Regularity problem for the class of operators satisfying (1) and (2).
 In \cite{MPT}, it is shown that the Regularity problem  in $L^p$ is solvable for elliptic operators whose matrix satisfies a (large) Carleson measure condition in domains
that are rougher than Lipschitz, assuming
solvability of a Dirichlet problem for the conjugate index $p'$.
From this, the authors are able to conclude solvability of $(R)_p$ for some $p>1$ 
for this class of operators on chord arc domains (or even weaker) domains.
On the domains they consider, it is necessary to work with the Haj\l{}asz-Sobolev space. This was identified in \cite{MouT}, where
results like those below were obtained for the Regularity problem for the Laplacian. In both papers, a corona-type decomposition of the domain is 
the foundational tool. In the case of \cite{MPT}, the decomposition of \cite{MT} has been modified in order
 to use the solvability of the Regularity problem from \cite{DPR} (for the weaker oscillation condition 
\eqref{carlM} on the matrix) as a black box in the Lipschitz subdomains they construct.

\medskip

\begin{theorem}[\cite{MPT}  $(D^*)_{p'} \Longrightarrow (R)_p$]\label{thm.regularity} Let $\Omega\subset\mathbb R^{n}$, $n\geq2$, be a bounded domain satisfying the corkscrew condition and with uniformly rectifiable boundary. Let $p\in(1,\infty)$, $p'$ its H\"older conjugate, and $ L=\mbox{div} (A\nabla)$, where $A$ 
satisfies \eqref{CCalt}.  Suppose that $(D^*)_{p'}$ is solvable in $\Omega$. Then $(R)_p$ is solvable in $\Omega$, and the constants in 
the norm bound depend only on ellipticity constants, $p$, $n$, the corkscrew constant, the uniform rectifiability constants, the constant in \eqref{CCalt}, and the 
$(D^*)_{p'}$ constant.
\end{theorem} 

As a corollary to solvability of the Dirichlet problem in chord arc domains, together with the perturbation theory for the Regularity problem which
holds on such domains, they obtain the following.

\begin{corollary}[Solvability of $(R)_p$ for some $p>1$]\label{cor.reg} Let $\Omega\subset\mathbb R^{n}$, $n\geq 2$, be a chord arc domain. Let $L=\mbox{div}(A\nabla)$ with $A$ a matrix as in the theorem. Then there exists $p>1$ such that $(R)_p$ is solvable in $\Omega$.	
\end{corollary}

In fact, the domains for which Corollary \ref{cor.reg} hold are more general: corkscrew with an AD-regular boundary and satisfying the IBPCAD condition. 

\smallskip

Finally, we give one reference for the perturbation results for the Regularity theory that are needed. There were several advances in perturbation 
in rougher domains than Lipschitz - the latest one can be found in \cite{DFMperturb} (and see also the references therein). 
In \cite{DFMperturb}, it is shown that
the solvability of the Dirichlet problem for an operator $L_1$ which is a {\it Carleson perturbation} of $L_0$, leads to a comparison of the 
nontangential maximal function of gradients of solutions with the same boundary values.

\begin{theorem}\label{regperturb}
Let $\Omega$ be a uniform domain  and let $L_0,\,L_1$ be two elliptic operators whose coefficients are real, non necessarily symmetric. Assume that the Dirichlet problem for the adjoint operator $L_1^*$ is solvable in $L^{q'}$.

If $L_1$ is a Carleson perturbation of $L_0$, then for any $f\in C_c(\partial \Omega) $, the two solutions,
$u_{0,f}$ and $u_{1,f}$ to the Dirichlet problems $L_0 u_{0,f}= 0$ and $L_1 u_{1,f} = 0$ with data $f$ verify
\begin{equation}\label{ntgrad}
\|\tilde{N}(\nabla u_{1,f})\|_{L^q(\partial \Omega,\sigma)} \leq C M \|\tilde{N}(\nabla u_{0,f})\|_{L^q(\partial \Omega,\sigma)},
\end{equation}
\end{theorem}

\medskip

As we defined earlier in this paper,  the disagreement (\cite{FKP}) between $A_0$ and $A_1$ is:
\begin{align}\label{DEFDF}
    \epsilon(X):=\sup_{Y\in B_{X}} |\mathcal{E}(Y)|,\ \ \mathcal{E}(Y):= A_0(Y)-A_1(Y).
\end{align}
and two operators are {Carleson perturbations} of one another when the following condition, which preserves $A_\infty$, holds.
\begin{equation}
\delta(X)\left[\sup_{Y\in B(X,\delta(X)/2)}|\epsilon(Y)|\right]^2\mbox{ is a Carleson measure}.
\end{equation}

\noindent The righthand side of \eqref{ntgrad} can be infinite, of course, but if it is finite then it must bound the corresponding nontangential maximal 
function estimate for gradients of solutions to $L_1$. 
From this result, one can conclude that if  $(R)_p$ is solvable for $L_0$ on a chord arc domain (or on the weaker domains considered in 
\cite{MPT}, where the notion of tangential derivatives is understand in terms of the Haj\l{}asz-Sobolev space), then so is $(R)_p$ for $L_1$.

\smallskip

The estimate in the statement of Theorem \ref{regperturb} may have implications for solvability of the Neumann problem, but this problem is still wide open even
for Lipschitz domains, except in the case of dimension two.


\section{Appendix - bounds of nontangential maximal function by square function.}
\setcounter{equation}{0}

In our treatment of the Dirichlet, Regularity and Neumann problems we have omitted proofs of Lemmas  \ref{l2}, \ref{ll3a}, \ref{S3:C7-alt1} and of the estimate \eqref{S3:C7:E00oozz}. For completeness we present here the main idea on how such results can be established. We have given them a unified treatment which is primarily based on \cite{DHM}. As we prefer to impose minimal possible assumptions on coefficients, we shall work here with the averaged version of the non-tangential maximal function $\widetilde{N}$ as defined in \eqref{NTMaxVar}. Recall that we always have $\widetilde{N}\le N$, with the opposite inequality $N\lesssim \widetilde{N}$  (with $\widetilde{N}$ using wider cones than $N$) holding in certain situations as well.
This holds for example for solutions $u$ of $Lu=0$ due to De Giorgi-Nash-Moser theory, for $\nabla u$ of such solutions if $|\nabla A|\lesssim \delta(X)^{-1}$ \cite{DPreg}, but fails to hold for gradients if the coefficients are just bounded and measurable.\medskip

The major innovation in the approach we present here is the use of an entire family of Lipschitz graphs on which the nontangential 
maximal function is large in lieu of a single graph constructed via a stopping time argument. 
This is necessary as we are using $L^2$ averages of solutions to define the nontangential maximal 
function and hence the knowledge of certain bounds for a solution on a single graph provides no 
information about the $L^2$ averages over interior balls. Our underlying domain we work on here is $\Omega=\mathbb R^n_+$. 

Let $u:\Omega=\{(x,t):t>0\}\to\mathbb R^N$ be a vector valued function such that $u\in L^2_{loc}(\Omega)$ and sufficient decay at infinity. Consider 
$$w(X)=\left(\fiint_{B_{{\delta(X)}/{2}}(X)}|u(Y)|^2\,d
Y\right)^{\frac{1}{2}}.$$
Then clearly, $\widetilde{N}(u)(Q)=N(w)(Q)$. Also clearly,  $w:\Omega\to \mathbb R$ be is continuous function with $w(x,t)\to 0$ as $t\to\infty$. For a constant $\nu>0$, define the set
\begin{equation}\label{E}
E_{\nu,a}:=\big\{x'\in\partial\Omega:\,N_{a}(w)(x')>\nu\big\}
\end{equation}
where, as usual, $a>0$ is a fixed background parameter denoting the aperture of cones used to define $N$. Also, 
consider the map $\hbar:\partial\Omega\to\mathbb R$ given at each $x'\in\partial\Omega$ by 
\begin{equation}\label{h}
\hbar_{\nu,a}(w)(x'):=\inf\left\{x_0>0:\,\sup_{z\in\Gamma_{a}(x_0,x')}w(z)<\nu\right\}
\end{equation}
with the convention that $\inf\varnothing=\infty$. 
We remark that $\hbar$ differs from the function $\tilde{\hbar}:\partial\Omega\to{\mathbb R}$ defined 
at each $x'\in\partial\Omega$ as
\begin{equation}\label{Eqqq-4}
\tilde{\hbar}_{\nu, a}(w)(x'):=\sup\left\{x_0 >0:\,\sup_{z\in\Gamma_{a}(x_0,x')}w(z)>\nu\right\}.
\end{equation}
The function $\tilde{\hbar}$ has been used in arguments for scalar equations 
(cf. \cite[pp.\,212]{KP} and \cite{KKPT2}). While there are clear similarities in the manner 
in which the functions $\hbar$ and $\tilde{\hbar}$ are defined, throughout this paper we prefer to use $\hbar$ 
as it works better here. 

At this point we observer that $\hbar_{\nu,a}(w,x')<\infty$ for all points $x'\in\partial\Omega$. 
This is due to the fact that we assume that the averages $w$ got zero as $t\to\infty$. 
It follows that $\hbar_{\nu,a}(w)(x')<\infty$.

\begin{lemma}\label{S3:L5} Fix two positive numbers $\nu,a$. Then the following properties hold.
\vglue2mm

\noindent (i)
The function $\hbar_{\nu, a}(w)$ is Lipschitz, with a Lipschitz constant $1/a$. That is,
\begin{equation}\label{Eqqq-5}
\left|\hbar_{\nu,a}(w)(x')-\hbar_{\nu,a}(w)(y')\right|\leq a^{-1}|x'-y'|
\end{equation}
for all $x',y'\in\partial\Omega$.

\vglue2mm

\noindent (ii)
Given an arbitrary $x'\in E_{\nu,a}$, let $x_0:=\hbar_{\nu,a}(w)(x')$. Then there exists a 
point $y=(y_0,y')\in\partial\Gamma_{a}(x_0,x')$ such that $w(y)=\nu$ and $\hbar_{\nu,a}(w)(y')=y_0$. 		
\end{lemma}

The proof is standard and can be found in \cite{DHM}.

\begin{lemma}\label{l6} For any $a>0$ there exists $b=b(a)>a$ and $\gamma=\gamma(a)>0$ such that the following holds. 
Having fixed an arbitrary $\nu>0$, for each point $x'$ from the set 
\begin{equation}\label{Eqqq-17}
\big\{x':\,N_{a}(w)(x')>\nu\mbox{ and }S_{b}(u)(x')\leq\gamma\nu\big\}
\end{equation}
there exists a boundary ball $R$ with $x'\in 2R$ and such that
\begin{equation}\label{Eqqq-18}
\big|w\big(\hbar_{\nu,a}(w)(z'),z'\big)\big|>\nu/{2}\,\,\text{ for all }\,\,z'\in R.
\end{equation}
\end{lemma}

Here $S_b$ is the square function associated with cones of aperture $b$. See  again \cite{DHM} for the proof.\medskip

Given a Lipschitz function $\hbar:{\mathbb{R}}^{n-1}\to{\mathbb{R}}$, denote by $M_\hbar$ the 
Hardy-Littlewood maximal function considered on the graph of $\hbar$. That is, 
given any locally integrable function $f$ on the Lipschitz surface 
$\Lambda_\hbar=\{(\hbar(z'),z'):\,z'\in{\mathbb R}^{n-1}\}$, define 
$(M_\hbar f)(x):=\sup_{r>0}\fint_{\Lambda_\hbar\cap B_r(x)}|f|\,d\sigma$ for each $x\in\Lambda_\hbar$. 

\begin{corollary}\label{S3:L6} 
Fix $a>0$.  and let $b,\,\gamma$ be as in Lemma~\ref{l6}. Then there exists a finite 
constant $C=C(n)>0$ with the property that for any $\nu>0$ and any point $x'\in E_{\nu,a}$ 
such that $S_{b}(u)(x')\leq\gamma\nu$ one has
\begin{equation}\label{Eqqq-23}
(M_{\hbar_{\nu,a}}w)\big(\hbar_{\nu,a}(x'),x'\big)\geq\,C\nu.
\end{equation}
\end{corollary}

\begin{proof} 
Fix a point $x'\in E_{\nu,a}$ where $S_{b}(u)(x')\leq\gamma\nu$. Lemma~\ref{l6} then guarantees the 
existence of a boundary ball $R$ with the property that $w(\hbar_{\nu,a}(w)(z'),z')>\nu/{2}$ for all 
$z'\in R$ and $x'\in 2R$. Granted this, it follows that 
\begin{equation}\label{Eqqq-24}
(M_{\hbar_{\nu,a}}w)\big(\hbar_{\nu,a}(w)(x'),x'\big)\geq\frac1{|2R|}\int_R w\big(\hbar_{\nu,a}(w)(z'),z'\big)\,dz'
\geq\frac{|R|}{|2R|}\frac{\nu}{{2}},
\end{equation}
as desired. 
\end{proof}

Now we aim to prove Lemma  \ref{l2} in the case $p>2$. The following lemma is crucial for its proof.

\begin{lemma}\label{S3:L8} 
Consider the elliptic operator $Lu=0$ with bounded coefficients given by a matrix $A$ such that $d\mu(X)=|\nabla (a_{n1},a_{n2},\dots,a_{nn})|^2t\,dX$ is a Carleson measure.
Then there exists $a>0$ with the following significance. Suppose $u$ is a weak solution of  $Lu=0$
in $\Omega={\mathbb{R}}^n_{+}$. Select $\theta\in[1/6,6]$ and, having picked $\nu>0$ arbitrary,  
let $\hbar_{\nu,a}(w)$ be as in \eqref{h}. Also, consider the domain 
$\mathcal{O}=\{(x_0,x')\in\Omega:\,x_0>\theta \hbar_{\nu,a}(x')\}$ with boundary  
$\partial\mathcal{O}=\{(x_0,x')\in\Omega:\,x_0=\theta \hbar_{\nu,a}(x')\}$. In this context, 
for any surface ball $\Delta_r=B_r(Q)\cap\partial\Omega$, with $Q\in\partial\Omega$ and $r>0$ 
chosen such that $\hbar_{\nu,a}(w)\leq 2r$ pointwise on $\Delta_{2r}$, 
one has
\begin{align}\label{TTBBMM}
\int_{\Delta_r}\big|u\big(\theta \hbar_{\nu,a}(w)(\cdot),\cdot\big)\big|^2\,dx' 
&\leq C(1+\|\mu\|_{Carl})\|S_b(u)\|_{L^2(\Delta_{2r})}
\|\tilde{N}_a(u)\|_{L^2(\Delta_{2r})}
\nonumber\\
&\quad+C\|S_b(u)\|^2_{L^2(\Delta_{2r})}+\frac{C}{r}\iint_{\mathcal{K}}|u(X)|^{2}\,dX.
\end{align}
Here $C=C(\lambda,\Lambda,n,N)\in(0,\infty)$ and $\mathcal{K}$ is a region inside $\mathcal{O}$ of diameter, 
distance to the boundary $\partial\mathcal{O}$, and distance to $Q$, are all comparable to $r$. 
Also, the parameter $b>a$ is as in Lemma~\ref{l6}, and the cones used to define the square and nontangential 
maximal functions in this lemma have vertices on $\partial\Omega$.

Moreover, the term $\displaystyle\iint_{\mathcal{K}}|u(X)|^2\,dX$ appearing 
in \eqref{TTBBMM} may be replaced by the quantity
\begin{equation}\label{Eqqq-25}
Cr^{n-1}|u(A_r)|^2+C\int_{\Delta_{2r}}S^2_b(u)\,d\sigma,
\end{equation}
where $A_r$ is any point inside $\mathcal{K}$ {\rm (}usually called a corkscrew point of $\Delta_r${\rm )}.
\end{lemma}

We postpone the proof until the very end and show consequences of this lemma.

\begin{lemma}\label{LGL} Let $L$ be as in Lemma \ref{S3:L8}.  Then for each $\gamma\in(0,1)$ there exists a constant $C(\gamma)>0$ 
such that $C(\gamma)\to 0$ as $\gamma\to 0$ and with the property that for each $\nu>0$ and 
each energy solution $u$ of $Lu=0$ there holds 
\begin{align}\label{eq:gl}
&\hskip -0.20in 
\left|\Big\{x'\in {{\mathbb R}}^{n-1}:\,\tilde{N}_a(u)>\nu,\,(M(S^2_b(u)))^{1/2}\leq\gamma\nu,\,
\big(M(S^2_b(u))M(\tilde{N}_a^2(u))\big)^{1/4}\leq\gamma\nu\Big\}\right|
\nonumber\\[4pt] 
&\hskip 0.50in
\quad\le C(\gamma)\left|\big\{x'\in{{\mathbb R}}^{n-1}:\,\tilde{N}_a(u)(x')>\nu/32\big\}\right|.
\end{align}
\end{lemma}

\begin{proof} 
To start, observe that $\big\{x'\in{{\mathbb R}}^{n-1}:\,\tilde{N}_a(u)(x')>\nu/32\}$ is an open 
subset of ${{\mathbb R}}^{n-1}$. When this set is empty or the entire Euclidean ambient, 
estimate \eqref{eq:gl} is trivial, so we focus on the case when the set in question is 
both nonempty and proper. Granted this, we may consider a Whitney decomposition $(\Delta_i)_{i\in I}$ 
of it, consisting of open cubes in ${\mathbb{R}}^{n-1}$. Let $F_\nu^i$ be the set appearing on the 
left-hand side of \eqref{eq:gl} intersected with $\Delta_i$. We may streamline the index set $I$ 
by retaining only those $i$'s for which $F_\nu^i\neq\varnothing$. Let $B_i$ be a ball 
of radius $r_i$ in ${\mathbb R}^n$ such that $\Delta_i\subset B_i\cap \{x_0=0\}$ and there 
exists a point $p'\in 2B_i\cap \partial{\mathbb R}^n_{+}$ with $\tilde{N}_a(u)(p')=N_a(w)(p')\leq\nu/32$. 
The existence of such point $p'$ is guaranteed by the very nature of the Whitney decomposition. 
Indeed, there exists a point near $\Delta_i$ not contained in the set 
$\{x'\in{{\mathbb R}}^{n-1}:\,\tilde{N}_a(u)(x')>\nu/32\}$.
 
This clearly implies that $w(z)\le\nu/32$ for all $z\in\Gamma_a(p')$. In particular, for all $x'\in \Delta_i$ 
we have $w(z)\le\nu/32$ for all $z\in\Gamma_a(x')\cap\Gamma_a(p')$, so we focus on estimating the size 
of $w(z)$ for $z\in\Gamma_a(x')\setminus\Gamma_a(p')$ with $z_0\geq 2r$. 
Since we also assume that for at least one $x'\in\Delta_i$ we have $M(S^2_b(u))(x')\leq(\gamma\nu)^2$, 
we may conclude that for sufficiently small $\gamma>0$ we have that for any $z\in\Gamma_a(x')$ with $z_0\ge 2r$ 
there is a point $\tilde{z}\in\Gamma_a(p')$ with 
\begin{equation}\label{Eqqq-33}
|z-\tilde{z}|\leq Cr_i\,\,\text{ and }\,\,|w(z)-w(\tilde{z})|\leq\nu/32.
\end{equation}
It follows that for all such $z$ we have $w(z)\le\nu/16$. Hence for all $x'\in \Delta_i$ we have
\begin{equation}\label{Eqqq-34}
\nu<\tilde{N}_a(u)(x')={N}_a(w)(x')=N_a^{2r}(w)(x'),
\end{equation}
where  $N_a^{2r}$ is the truncated nontangential maximal function at height $2r$. In particular this also implies
\begin{equation}\label{eq:hbound}
\hbar_{\nu,a}(w)\leq 2r_i\,\,\text{ pointwise on }\,\,\Delta_i.
\end{equation}

Let us also note that we can find a point $q$ (specifically, a corkscrew point for $12\Delta_i$) 
with distance to $\Delta_i$ and the boundary equal to $12r_i$ such that $w(q)\leq\nu/16$. 
When $h\lesssim r_i$ since $u$ vanishes above height $h$ we might actually take $q$ such that $w(q)=0$.

As $w$ is the $L^2$ average of $|u|$, De Giorgi-Nash-Moser theory implies that
\begin{equation}\label{Eqqq-35}
|u(q)|\leq w(q)\le\nu/16.
\end{equation}

Next, consider $\tilde{u}:=u-u(q)$.  Then ${L}\tilde{u}=0$, hence $\tilde{u}$ still 
solves $Lu=0$ and $\tilde{u}({q})=0$. Denote by $\tilde{w}$ the $L^2$ 
averages of $|\tilde u|$. For all $x'\in F^i_{\nu}$ we have
\begin{equation}\label{Eqqq-36}
N_a^{2r}(\tilde{w})(x')\geq N_a^{2r}({w})(x')-|u({q})|\geq\nu-\nu/16>\nu/2.
\end{equation}
With $\hbar:=\hbar_{\nu,a}(w)$ and for $M_\hbar$ defined on the graph of $\hbar$ 
in Corollary~\ref{S3:L6} we see that Corollary~\ref{S3:L6} 
applied to $\tilde{u}$ implies\footnote{Technically $\tilde{u}\in W^{1,2}_{\rm loc}(\Omega)$ 
is not an energy solution, but in the proof the smallness of the solution is only needed above 
a certain distance from the boundary. In our case we obviously have 
$\tilde{w}(z)\leq w(z)+|u({q})|\leq\nu/8$ for points $z$ whose distance to the boundary 
exceeds $2r_i$ which suffices for our purposes.}  
\begin{equation}\label{Eqqq-37}
M_\hbar\left(\tilde{w}\chi_{4B_i}\right)\big(\hbar(x'),x'\big)\geq C(n)\nu.
\end{equation}
Here we are allowed to apply the cutoff function $\chi_{4B_i}$ since values of 
$\tilde{w}$ are small above the height $2r$, hence this places a bound on the distance 
and the diameter of the boundary ball $R$ constructed in Corollary~\ref{S3:L6} from the 
point $x'$ (both are bounded by $\lesssim r_i$). Thus, by the maximal function theorem
\begin{align}\label{UHB}
|F_\nu^i| &\leq\frac{C}{\nu^2}\int_{4\Delta_i}\big(M_\hbar(\tilde{w}\chi_{4B_i})\big)^{2}\big(\hbar(x'),x'\big)\,dx'
\nonumber\\[4pt]
&\leq\frac{C}{\nu^2}\int_{4\Delta_i}\tilde{w}^2(\hbar(x'),x')\,dx'.
\end{align}

At this stage, we bring in the following lemma. proof of which is again in \cite{DHM},

\begin{lemma}\label{Lw-u} 
For any surface ball $\Delta$, if $a>0$ and $\hbar=\hbar_{\nu,a}(w)$ then 
\begin{equation}\label{Eqqq-38}
\int_{\Delta}\tilde{w}^2(\hbar(x'),x')\,dx'\leq C\int_{1/6}^{6}\int_{3\Delta} 
\big|\tilde{u}(\theta \hbar(x'),x')\big|^2\,dx'\,d\theta.
\end{equation}
\end{lemma}

Hence, we have (taking $a>0$ as in Lemma~\ref{S3:L8})
\begin{equation}\label{eq:Fnu}
|F_\nu^i|\leq\frac{C}{\nu^2}\int_{1/6}^{6}\int_{12\Delta_i}\big|\tilde{u}(\theta \hbar(x'),x')\big|^2\,dx'\,d\theta.
\end{equation}
We apply the conclusion in Lemma~\ref{S3:L8} (in the version recorded in the very last 
part of its statement) to the solution $\tilde{u}$. This gives
\begin{align}\label{NEWEST}
&\hskip 0.00in
\int_{1/6}^6\int_{12\Delta_i}|\tilde{u}(\theta \hbar(x'),x')|^2\,dx'\,d\theta 
\\[4pt]
&\hskip 0.20in
\leq C(1+\|\mu\|^{1/2}_{\mathcal C})\|S_{b}({u})\|_{L^2(24\Delta_i)}
\|{N}_{a}(\tilde{w})\|_{L^2(24\Delta_i)}
\nonumber\\[4pt]
&\hskip 0.20in
\quad+C\|S_{b}(u)\|^2_{L^2(24\Delta_i)}+Cr^{n-1}|\tilde{u}(q)|^2
\nonumber\\[4pt]
&\hskip 0.20in
\leq C(1+\|\mu\|^{1/2}_{\mathcal C})\|S_{b}(u)\|_{L^2(24\Delta_i)}
\|{N}_{a}({w}+w({q}))\|_{L^2(24\Delta_i)}
\nonumber +C\|S_{b}(u)\|^2_{L^2(24\Delta_i)}.
\end{align}
Observe that we have 
dropped the term $Cr^{n-1}|\tilde{u}(q)|^2$ as we have arranged previously that 
$\tilde{u}(q)=0$. Since $F_\nu^i\neq\varnothing$ and $|w({q})|\leq\nu/16$ the 
first term of the last line of \eqref{NEWEST} may be bounded by
\begin{align}
&\hskip -0.20in
C|24\Delta_i|\left(\fint_{24\Delta_i} S_b^2(u)dx'\right)^{1/2}
\left[\left(\fint_{24\Delta_i} N_a^2(u)dx'\right)^{1/2}+\frac\nu{16}\right]
\nonumber\\[4pt]
&\hskip 0.20in
\leq C|24\Delta_i|\left[\big(M(S^2_b(u))(x')M(\tilde{N}_a^2(u))(x')\big)^{1/2} 
+\frac\nu{16}M\big(S^2_b(u)\big)(x')^{1/2}\right]
\nonumber\\[4pt]
&\hskip 0.20in
\leq C|24\Delta_i|(\gamma^2+\gamma/16)\nu^2=C(\gamma)|\Delta_i|\nu ^2.
\end{align}
Here $x'\in F_\nu^i$ is a point where we use the assumptions for the set on 
the left-hand side of \eqref{eq:gl}. Also, we have used that $|24\Delta_i|\lesssim|\Delta_i|$ by the 
doubling property of the Lebesgue measure. The estimate for the very last term of \eqref{NEWEST} is analogous. 
By design, we have $C(\gamma)\to 0$ as $\gamma\to 0$. Using this back in \eqref{eq:Fnu} we obtain
\begin{equation}\label{Eqqq-39}
|F_\nu^i|\leq C'(\gamma)|\Delta_i|.
\end{equation}
Summing over all $i$ we obtain \eqref{eq:gl}, as desired.
\end{proof}

What we have just established is an example of so-called good-lambda inequalities which are key in this theory. We immediately get the following:

\begin{proposition}\label{S3:C7} Let $L$ be as in Lemma \ref{S3:L8}.  
The for any $p>2$ and $a>0$ there exists an integer $m=m(a)\ge 2$ and a finite constant 
$C>0$ such that for any energy solution $u$ 
of $Lu=0$ in $\Omega$ we have:
\begin{equation}\label{S3:C7:E00ooqq}
\|\tilde{N}_a(u)\|_{L^{p}({{\mathbb R}}^{n-1})}\le C\|S_a(u)\|_{L^{p}({{\mathbb R}}^{n-1})}.
\end{equation}
\end{proposition}

\begin{proof} Multiply the good-$\lambda$ inequality \eqref{eq:gl} by $\nu^{p-1}$ and integrate in $\nu$ 
over the interval $(0,\infty)$. This implies:

\begin{align}\label{eq:glzzz}
&\int_{\partial\mathbb R^n_+}N_a(u)^pdx \le C'(\gamma) \int_{\partial\mathbb R^n_+}N_a(u)^pdx+K
\int_{\partial\mathbb R^n_+}(M(S^2_b(u)))^{p/2}dx\\&+K\int_{\partial\mathbb R^n_+}(M(S^2_b(u))M(\tilde{N}_a^2(u))\big)^{p/4}dx.\nonumber
\end{align}
Here $C'(\gamma)\to 0$ as $\gamma\to 0$. Thus it is possible to pick $\gamma>0$ for which $C'(\gamma)<1/2$ and hide $C'(\gamma) \int_{\partial\mathbb R^n_+}N_a(u)^pdx$ on the lefthand side. For $p/2>1$ the maximal function is bounded on $L^{p/2}$ and hence \eqref{eq:glzzz} implies
\begin{align}\label{eq:glzzzx}
&\int_{\partial\mathbb R^n_+}N_a(u)^pdx \le K'
\int_{\partial\mathbb R^n_+}S_b(u)^pdx.
\end{align}
From this our claim follows. The proof that above actually holds for all $p>0$ requires a more sophisticated version of \eqref{eq:gl} which is localised.   
The corresponding local version of the estimate \eqref{S3:C7:E00ooqq} for $p>2$ is the necessary ingredient for what  is otherwise a purely abstract, real-variable argument that then extends the estimate to all $p>0$. Further details can be found in \cite{FSt}.
\end{proof}

In a similar spirit Lemma \ref{ll3a} can be established using the following Lemma from \cite[Lemma 3.4]{DPreg} :

\begin{lemma}\label{S3:L8bb} 
Consider the elliptic PDE $Lu=0$ with coefficients satisfying Carleson condition \eqref{CCalt}, let $v=\nabla u$ and let $w$ be the $L^2$ averages of $v$. 
Then there exists $a>0$ with the following significance.  Select $\theta\in[1/6,6]$ and, having picked $\nu>0$ arbitrary,  
let $h_{\nu,a}(w)$ be as in \eqref{h}. Also, consider the domain 
$\mathcal{O}=\{(x_0,x')\in\Omega:\,x_0>\theta h_{\nu,a}(x')\}$ with boundary  
$\partial\mathcal{O}=\{(x_0,x')\in\Omega:\,x_0=\theta h_{\nu,a}(x')\}$. In this context, 
for any surface ball $\Delta_r=B_r(Q)\cap\partial\Omega$, with $Q\in\partial\Omega$ and $r>0$ 
chosen such that $h_{\nu,a}(w)\leq 2r$ pointwise on $\Delta_{2r}$, 
one has for an arbitrary $\vec{c}=(c_1,c_2,\dots,c_{n})\in\mathbb R$:
\begin{align}\label{TTBBMMww}
\int_{1/6}^6\int_{\Delta_r}\big|v\big(\theta h_{\nu,a}(w)(\cdot),\cdot\big)-\vec{c}\big|^2\,dx d\theta
&\leq C(1+\|\mu\|^{1/2}_{\mathcal C})\|S_{b}(v)\|_{L^p(\Delta_{2r})}
\|\tilde{N}_{2,a}(v-\vec{c})\|_{L^p(\Delta_{2r})}
\nonumber\\
&\hskip-4cm +C\|\mu\|_{\mathcal C}^{1/2}\|\tilde{N}_{2,a}(v-\vec{c})\|^2_{L^p(\Delta_{2r})}+C\|S_{b}(v)\|^2_{L^p(\Delta_{2r})}+\frac{C}{r}\iint_{\mathcal{K}}|v-\vec{c}|^{2}\,dX.
\end{align}
Here $C=C(\lambda,\Lambda,p,n)\in(0,\infty)$ and $\mathcal{K}$ and $\mathcal{O}$ are as in Lemma \ref{S3:L8}.
\end{lemma}

Finally, Lemma \ref{S3:C7-alt1} follows from \cite[Lemma 5.1]{DHP}
\begin{lemma}\label{S3:L8-alt1zz} 
Let $\Omega={\mathbb R}^n_+$ and let $Lu=\mbox{div}_x(B_\parallel\nabla_x u)+u_{tt}$ be a block-form operator with bounded measurable coefficients.
Suppose $V=(v_1,v_2,\dots,v_{n-1})$ is a weak solution of \eqref{system} in $\Omega$  For a fixed (sufficiently large) $a>0$, consider an arbitrary Lipschitz function $\hbar:{\mathbb R}^{n-1}\to \mathbb R$ such that
\begin{equation}
\|\nabla \hbar\|_{L^\infty}\le 1/a,\qquad \hbar(x)\ge 0\text{ for all }x\in{\mathbb R}^{n-1}.\label{hbarprop}
\end{equation}
Then for sufficiently large $b=b(a)>0$ we have the following. For an arbitrary surface ball $\Delta_r\subset{\mathbb R}^{n-1}$ of radius $r$ such that at least one point of $\Delta_r$
the inequality $\hbar(x)\le 2r$ holds we have the following estimate for all $m=1,2,\dots,n-1$ and an arbitrary $\vec{c}=(c_1,c_2,\dots,c_{n-1})\in\mathbb R$:
\begin{align}\label{TTBBMMxx}
&\sum_{m<n}\int_{1/6}^6\int_{\Delta_r}\big|v_m\big(x,\theta\hbar(x)\big)-c_m\big|^2\,dx\,d\theta
\leq C\Big[\|S_b(V)\|_{L^2(\Delta_{2r})}
\|{N}_a(V-\vec{c})\|_{L^2(\Delta_{2r})}
\nonumber\\+&\|\mu\|_{Carl}(\|{N}_a(V-\vec{c})\|_{L^2(\Delta_{2r})}^2+\|N_a(V)\|^2_{L^2(\Delta_{2r})})+\|S_b(V)\|^2_{L^2(\Delta_{2r})}+\frac{1}{r}\iint_{\mathcal{K}}|V-\vec{c}|^{2}\,dX\Big],
\end{align}
for some $C\in(0,\infty)$ that only depends on $a,\lambda,\Lambda,n$. Here $d\mu(X)=|\nabla_x B_\parallel(X)|^2 t\,dX$ is the Carleson measure and $\mathcal{K}$ and $\mathcal{O}$ are as in Lemma \ref{S3:L8}.
\end{lemma}

\noindent{\it Remark.} Let us explain the role of the vector $\vec{c}$ in Lemmas \ref{S3:L8bb}-\ref{S3:L8-alt1zz}. Note that Lemma \ref{S3:L8} does not contain it. This is due to the fact that in Lemma \ref{S3:L8}  if $u$ solves PDE $Lu=0$ then so does $u-c$ and hence Lemma \ref{S3:L8} must also hold for it. The difference in the other two lemmas is that $\nabla u$ and $\nabla u-\vec{c}$ do not solve the same PDE system and  we need a claim that works for both. The particular place this is used is just below \eqref{Eqqq-35} where $\tilde{u}$ is defined with a particular property that its avarage at a corkscrew point is zero and then Lemma \ref{S3:L8} is applied to $\tilde{u}$. This is a small gap we have not realized in our the original paper \cite{DPR} and hence this lemma there is formulated only for $\nabla u$.\medskip

Finally, we prove Lemma \ref{S3:L8}. Proofs of Lemmas \ref{S3:L8bb}-\ref{S3:L8-alt1zz} are similar and hence we omit them.\medskip

\noindent{\it Proof of Lemma \ref{S3:L8}.}
Let $\Delta_r$ be as in the statement of our Lemma. and assume that $(q,0)$ in the center of our ball. Let $\zeta$ be a smooth cutoff function of the form $\zeta(x,t)=\zeta_{0}(t)\zeta_{1}(x)$ where
\begin{equation}\label{Eqqq-27}
\zeta_{0}= 
\begin{cases}
1 & \text{ in } (-\infty, r_0+r], 
\\
0 & \text{ in } [r_0+2r, \infty),
\end{cases}
\qquad
\zeta_{1}= 
\begin{cases}
1 & \text{ in } \Delta_{r}(q), 
\\
0 & \text{ in } \mathbb{R}^{n}\setminus \Delta_{2r}(q)
\end{cases}
\end{equation}
and
\begin{equation}\label{Eqqq-28}
r|\partial_{t}\zeta_{0}|+r|\nabla_{x}\zeta_{1}|\leq c
\end{equation}
for some constant $c\in(0,\infty)$ independent of $r$. Here 
$r_0=6\sup_{x\in \Delta_r(q)}\hbar(x)$. Observe that our assumptions imply that
$$0\le r_0-\theta\hbar(x)\le  r_0 \lesssim r,\qquad \mbox{for all }x\in \Delta_{2r}(q),$$
for $\theta\in (1/6,6)$.

Our goal is to control the $L^2$ norm of $u$.  We proceed to estimate
\begin{align}
&\hskip -0.20in
\int_{\Delta_{r}(q)}u(x,\theta\hbar(x))^2\,dx \le \mathcal I:=\int_{\Delta_{2r}(q)}u(x,\theta\hbar(x))^2\zeta(x,\theta\hbar(x))\,dx
\nonumber\\[4pt]
&\hskip 0.70in
=-\iint_{\mathcal S(q,r,r_0,\theta\hbar)}\partial_{t}\left[u(x,t)^2\zeta(x,t)\right]\,dt\,dx,
\nonumber
\end{align}
where $\mathcal S(q,r,r_0,\theta\hbar)=\{(x,t):x\in \Delta_{2r}(q)\mbox{ and }\theta\hbar(x)<t<r_0+2r\}$. Hence:

\begin{align}\nonumber
&\hskip 0.10in
\mathcal I \le-2\iint_{\mathcal S(q,r,r_0,\theta\hbar)}u\partial_{t}u\zeta\,dt\,dx  
\\[4pt]
&\hskip 0.70in
\quad-\iint_{\mathcal S(q,r,r_0,\theta\hbar)}u^2(x,t)\partial_{t}\zeta\,dt\,dx
=:\mathcal{A}+IV.\label{u6tg}
\end{align}
We further expand the term $\mathcal A$ as a sum of three terms obtained 
via integration by parts with respect to $t$ as follows:
\begin{align}\label{utAA}
\mathcal A &=-2\iint_{\mathcal S(q,r,r_0,\theta\hbar)}u\partial_{t} 
u(\partial_{t}t)\zeta\,dt\,dx 
\nonumber\\[4pt]
&=2\iint_{\mathcal S(q,r,r_0,\theta\hbar)}\left|\partial_{t}u\right|^{2}t\zeta\,dt\,dx 
+2\iint_{\mathcal S(q,r,r_0,\theta\hbar)}u(\partial^2_{tt}u)t\zeta\,dt\,dx 
\nonumber\\[4pt]
&\quad +2\iint_{\mathcal S(q,r,r_0,\theta\hbar)}u\partial_{t}u\,t\partial_{t}\zeta\,dt\,dx
=:I+II+III.
\end{align}

We start by analyzing the term $II$. As the $u$ solve the PDE $Lu=0$ we see that
$$\partial^2_{tt}u=\partial_t\left(\frac{a_{nn}\partial_t u}{a_{nn}}\right)=\frac{\partial_t(a_{nn}\partial_tu)}{a_{nn}}+\partial_t\left(\frac1{a_{nn}}\right)a_{nn}\partial_tu$$
$$=-\frac{\partial_ta_{nn}}{a_{nn}}\partial_tu-\sum_{(i,j)\ne(n,n)}\frac{\partial_i(a_{ij}\partial_ju)}{a_{nn}}.$$
When $i<n$ in the sum above we are happy. Otherwise we write $\partial_t(a_{nj}\partial_ju)$ as
$$\sum_{j<n}\left[\partial_j(a_{nj}\partial_tu)+(\partial_ta_{nj})\partial_j u-(\partial_ja_{nj})\partial_tu\right].$$

In turn, this permits us to write the term $II$ 

\begin{align}
II &=-2\sum_{i<n}\iint_{\mathcal S(q,r,r_0,\theta\hbar)}(a_{nn})^{-1}u\partial_{i}\left({a}_{ij}\partial_{j}u\right)t\zeta\,dt\,dx\nonumber\\[4pt]
&\quad-2\sum_{j<n}\iint_{\mathcal S(q,r,r_0,\theta\hbar)}(a_{nn})^{-1}u\partial_{j}\left({a}_{nj}\partial_{t}u\right)t\zeta\,dt\,dx
\nonumber\\[4pt]
&\quad+\mbox{error terms}=:II_1+II_2+\mbox{error terms},
\nonumber
\end{align}
where the error terms are all bounded by 
$$\iint_{\mathcal S(q,r,r_0,\theta\hbar)}|\nabla A||u||\nabla u|t\zeta\,dt\,dx\lesssim \|\mu\|_{Carl}
\|S_b(u)\|_{L^2(\Delta_{2r})}\|{N}_a(u)\|_{L^2(\Delta_{2r})}^2.$$
using the Carleson condition for $A$ and the Cauchy-Schwarz inequality.

Observe that the two main terms $II_1$ and $II_2$ are of the same type which motivates us to define $b_{ij}=a_{ij}$ when $i,j<n$ and $b_{in}=a_{in}+a_{ni}$ to obtain
\begin{align}
II_1+II_2&=-2\sum_{i<n}\iint_{\mathcal S(q,r,r_0,\theta\hbar)}(a_{nn})^{-1}u\partial_{i}\left({b}_{ij}\partial_{j}u\right)t\zeta\,dt\,dx\nonumber\\[4pt]
&=2\sum_{i,j<n}\iint_{\mathcal S(q,r,r_0,\theta\hbar)}b_{ij}\partial_i(a_{nn}^{-1})u(\partial_j u)\,t\zeta\,dt\,dx
\nonumber\\[4pt]
&+2\sum_{i<n}\iint_{\mathcal S(q,r,r_0,\theta\hbar)}(a_{nn}^{-1})b_{ij}(\partial_iu)(\partial_ju)\,t\zeta\,dt\,dx
\nonumber\\[4pt]
&-2\sum_{i,j<n}\iint_{\mathcal S(q,r,r_0,\theta\hbar)}(a_{nn}^{-1})b_{ij}u(\partial_j u)t(\partial_i \zeta)\,dt\,dx 
\nonumber\\[4pt]
&\quad-2\sum_{i>0}\int_{\partial\mathcal S(q',r,r_0,\theta\hbar)}(\mbox{boundary terms})t\zeta\nu_i\,dS
\nonumber\\[4pt]
&=:II_{3}+II_{4}+II_{5}+II_{6}.\label{TFWW}
\end{align}
Here we integrated by parts w.r.t. $\partial_i$. The term $II_3$ can again be considered to be an error term with the same estimate as given above.
The boundary integral (term $II_6$) vanishes everywhere except on the graph of the function $\theta\hbar$ which implies that
\begin{align}
|II_6|&\le C\sum_{i,j<n}\int_{\Delta_{2r}(q)}|u(x,\theta\hbar(x))(\nabla u)(x,\theta\hbar(x))\hbar(x)\zeta(x,\theta\hbar(x)) \nu_i|dS.\nonumber\\
&\le \frac12\int_{\Delta_{2r}(q)}u(x,\theta\hbar(x))^2\zeta(x,\theta\hbar(x))\,dx\nonumber\\&\quad+C'
\int_{\Delta_{2r}(q)}|\nabla u(x,\theta\hbar(x))|^2|\hbar(x)|^2\,dx
=\frac12\mathcal I+II_7.
\end{align}
Here we have used the AG inequality.
 We can hide the term $\frac12\mathcal I$ on the lefthand side of \eqref{u6tg}, while the second term after integrating $II_7$ in $\theta$ becomes:
\begin{align}
\int_{1/6}^6|II_7|\,d\theta&\le C\int_{1/6}^6\int_{\Delta_{2r}(q)}|\nabla u(x,\theta\hbar(x))|^2 |\hbar(x)|^2dxd\theta.\nonumber\\
&\lesssim\iint_{\Delta_{2r}(q)\times[0,r_0]}|\nabla u|^2t\,dt\,dx\lesssim\|S_b(u)\|^2_{L^2(\Delta_{2r})}.
\end{align}

Some of the remaining (solid integral) terms that are of the same type we estimate together. Firstly, we have 
\begin{equation}\label{Eqqq-29}
|I+II_4|\lesssim\|S_b(u)\|^2_{L^2(\Delta_{2r})}.
\end{equation}
Here, the estimate holds even if the square function truncated at a hight $O(r)$.
Next, since $r|\nabla\zeta|\le c$, if the derivative falls on the cutoff function $\zeta$ we have
\begin{align}\label{TDWW}
|II_5+III| &\lesssim \iint_{[0,2r]\times \Delta_{2r}}\left|\nabla u\right||u|\frac{t}{r}\,dt\,dx
\nonumber\\[4pt]
&\le \left(\iint_{[0,2r]\times \Delta_{2r}}|u|^{2}\frac{t}{r^{2}}\,dt\,dx\right)^{1/2} 
\|S^{2r}_b(u)\|_{L^2(\Delta_{2r})} 
\nonumber\\[4pt]
&\lesssim\|S_b(u)\|_{L^2(\Delta_{2r})}\|\tilde{N}_a(u)\|_{L^2(\Delta_{2r})}.
\end{align}

Finally, the interior term $IV$, which arises from the fact that $\partial_{0}\zeta$ vanishes on the set
$(-\infty,r_0+r)\cup(r_0+2r,\infty)$ may be estimated as follows:
\begin{equation}\label{Eqqq-31}
|IV|\lesssim\frac{1}{r}\iint_{\Delta_{2r}(q)\times [r_0+r,r_0+2r]}|u|^{2}\,dt\,dx.
\end{equation}
We put together all terms and integrate in $\theta$. The above analysis ultimately yields \eqref{TTBBMM}.
It is worth noting that we have only used the Carleson condition for the coefficients of the last row of the matrix $A$ and hence if for example $A$ is a block form matrix then no assumption beyond boundedness is needed for coefficients with indices $1\le i,j<n$.

Finally, the last claim in the statement of the lemma that we can use \eqref{Eqqq-25} on the righthand side instead of the solid integral is a consequence of the Poincar\'e's inequality.
\qed

\end{document}